\documentclass[final,hidelinks,onefignum,onetabnum]{siamart220329}

\usepackage{amsfonts}
\usepackage{graphicx}
\usepackage{epstopdf}
\usepackage{algorithm}
\usepackage{algorithmic}
\usepackage{amsmath}
\usepackage[T1]{fontenc}
\usepackage{fixmath}
\usepackage{makecell}
\usepackage{subfigure}
\usepackage{wrapfig}
\ifpdf
  \DeclareGraphicsExtensions{.eps,.pdf,.png,.jpg}
\else
  \DeclareGraphicsExtensions{.eps}
\fi

\newsiamremark{remark}{Remark}
\newsiamremark{hypothesis}{Hypothesis}
\crefname{hypothesis}{Hypothesis}{Hypotheses}
\newsiamthm{claim}{Claim}

\headers{Parareal for Particle-in-Fourier}{Parareal for Particle-in-Fourier}

\title{Error Analysis and Parallel Scaling Study of A Parareal Parallel-in-Time Integration Algorithm for Particle-in-Fourier Schemes}

\author{Sriramkrishnan Muralikrishnan\thanks{J{\"u}lich Supercomputing Centre, Forschungszentrum J{\"u}lich GmbH, 52428 Germany.
  (\email{s.muralikrishnan@fz-juelich.de}, \email{r.speck@fz-juelich.de}).}
\and Robert Speck\footnotemark[1]}

\usepackage{amsopn}

\newcommand{\beq} {\begin{equation}}
\newcommand{\eeq} {\end{equation}}
\newcommand{\bdm} {\begin{displaymath}}
\newcommand{\edm} {\end{displaymath}}
\newcommand{\bit}{\begin{itemize}}
\newcommand{\eit}{\end{itemize}}
\newcommand{\bde}{\begin{description}}
\newcommand{\ede}{\end{description}}
\newcommand{\bce}{\begin{center}}
\newcommand{\ece}{\end{center}}
\newcommand{\ben} {\begin{enumerate}}
\newcommand{\een} {\end{enumerate}}
\newcommand{\bea} {\begin{eqnarray}}
\newcommand{\eea} {\end{eqnarray}}
\newcommand{\barr} {\begin{array}}
\newcommand{\earr} {\end{array}}
\newcommand{\bean} {\begin{eqnarray*}}
\newcommand{\eean} {\end{eqnarray*}}
\newcommand{\edoc} {\end{document}}

\newcommand{\DD}[2] {\ensuremath{\frac{d {#1}}{d {#2}}}}

\newcommand{\eval}[2][\right]{\relax
  \ifx#1\right\relax \left.\fi#2#1\rvert}

\providecommand{\norm}[1]{\left\Vert#1\right\Vert}

\newcommand{\mc}[1]{\mathcal{#1}}

\newcommand{\mb}[1]{\mathbold{#1}}

\newcommand{\LRp}[1]{\left( #1 \right)}
\newcommand{\LRs}[1]{\left[ #1 \right]}

\newcommand{\LRc}[1]{\left\{ #1 \right\}}

\newcommand{\bigO}{\mathcal{O}}

\newcounter{tablecell}

\newcommand{\figlab}[1]{\label{fig:#1}}
\newcommand{\eqnlab}[1]{\label{eq:#1}}
\newcommand{\theolab}[1]{\label{theo:#1}}

\newcommand{\lemlab}[1]{\label{lem:#1}}

\newcommand{\figref}[1]{\ref{fig:#1}}
\newcommand{\theoref}[1]{\ref{theo:#1}}

\newcommand{\lemref}[1]{\ref{lem:#1}}
\newcommand{\eqnref}[1]{\eqref{eq:#1}}

\newcommand{\seclab}[1]{\label{sect:#1}}
\newcommand{\secref}[1]{\ref{sect:#1}}

\renewcommand{\u}{u}

\newcommand{\rhok}{\rho_{\mb{k}}}
\renewcommand{\v}{v}

\newcommand{\kb}{\mb{k}}

\newcommand{\q}{q}

\newcommand{\qb}{{\mb{\q}}}

\newcommand{\Sb}{\mb{S}}
\newcommand{\Sk}{S_{\mb{k}}}
\newcommand{\Sbk}{\Sb_{\mb{k}}}

\newcommand{\ub}{\mb{\u}}

\newcommand{\vb}{\mb{\v}}

\newcommand{\xb}{\mb{x}}

\newcommand{\Fb}{{\mb{F}}}

\newcommand{\B}{\mb{B}}

\newcommand{\Bb}{\mb{B}}

\newcommand{\Eb}{\mb{E}}
\newcommand{\deltab}{\mb{\delta}}

\newcommand{\Ebk}{\Eb_{\mb{k}}}

\newcommand{\eb}{\mb{e}}

\newcommand{\Lb}{\mb{L}}
\newcommand{\Gb}{\mb{G}}

\newcommand{\Pb}{\mb{P}}

\newcommand{\PbH}{\Pb^H}
\newcommand{\Pbh}{\Pb_h}
\newcommand{\PbhH}{\Pb_h^H}
\newcommand{\nb}{\mb{n}}

\graphicspath{{./figures/}}
\begin{document}

\maketitle

\begin{abstract}
    We propose a parareal based time parallelization scheme in the phase-space for the particle-in-Fourier (PIF) discretization of the Vlasov-Poisson 
    system used in kinetic plasma simulations. We use PIF with a coarse tolerance for the nonuniform fast Fourier transforms, or the standard particle-in-cell scheme, combined with temporal coarsening, as coarse propagators. This is different from the typical spatial coarsening of particles and/or Fourier modes for parareal, which are not possible or effective for PIF schemes. We perform an error analysis of the algorithm and verify the results numerically with Landau damping, two-stream instability, and Penning trap test cases
    in 3D-3V. We also implement the space-time parallelization of the PIF schemes in the open-source, performance-portable library IPPL and conduct scaling studies up to 1536 A100 GPUs on the JUWELS booster supercomputer. The space-time parallelization utilizing the parareal algorithm for the time parallelization provides up to $4-6$ times speedup compared to spatial parallelization alone and achieves a push rate of around 1 billion particles per second for the benchmark plasma mini-apps considered.  
\end{abstract}

\begin{keywords}
Particle-in-Fourier, Parareal, GPUs, Particle-in-cell, Parallel-in-time, Plasma physics 
\end{keywords}

\begin{MSCcodes}
35Q83, 65M75, 82D10
\end{MSCcodes}

\section{Introduction}

Particle-in-cell (PIC) schemes have been the method of choice for the simulation of kinetic plasmas since their inception \cite{hockney2021computer,birdsall2004plasma,dawson1983particle}. The simplicity, ease of parallelization and robustness for a wide variety of physical scenarios have contributed to the success of these schemes over the years. However, the presence of a grid in the standard PIC schemes leads to aliasing, as the particles live in the continuous phase-space and the modes
which are not resolved by the grid get aliased onto the lower frequency modes. This in turn results in a
numerical instability known as finite grid instability and loss of energy conservation when the Debye
length is not resolved \cite{langdon1970theory,huang2016finite,birdsall2004plasma}.

Previous efforts to improve conservation properties of the standard explicit PIC schemes 
fall into several categories. Earlier works such as \cite{lewis1970energy,eastwood1991virtual} 
discretize the Lagrangian formulation 
in \cite{low1958lagrangian} towards improving 
energy conservation. In \cite{chen2011energy, markidis2011energy, lapenta2017exactly, chen2020semi, pinto2022semi} 
the authors take advantage of fully implicit or 
semi-implicit time integration schemes to enforce exact energy conservation. The other category of energy conserving schemes are the structure preserving
geometric PIC schemes based on a variational 
formulation \cite{squire2012geometric,jianyuan2018structure,campos2022variational,campos2024variational} and 
discretization of the underlying Hamiltonian structure \cite{kraus2017gempic,he2016hamiltonian}. Structure preserving integrators for the Landau collision operator have been proposed in \cite{kraus2017metriplectic}. 
It should be noted that aliasing still occurs in these energy conserving schemes due 
to the presence of a grid but the effects are very much mitigated compared to the standard explicit PIC schemes \cite{barnes2021finite,xiao2019structure}.
Although many of these schemes have excellent long time stability and conservation properties, they differ significantly from the standard PIC framework and hence may not be as intuitive or easy to transition to from an application and implementation point of view.

The problems with aliasing in PIC schemes can be avoided if we 
interpolate from the particles directly to a truncated Fourier basis, solve the field equations in the Fourier space and interpolate the fields from the Fourier space back to the particle locations. This scheme, which is known
as particle-in-Fourier (PIF), is mentioned even in the very early literature on kinetic plasma simulations \cite{langdon1970theory} to study the effect of
using a spatial grid in PIC schemes. In \cite{evstatiev2013variational} the authors introduced a variational formulation of particle algorithms for kinetic plasma simulations based on 
Low's Lagrangian and recovered this scheme when a truncated Fourier basis is used. They also showed that PIF is a geometric structure-preserving scheme as it is derived from a Lagrangian, based on the principle of least action. In addition, they prove that in finite dimensions it is the only scheme which can simultaneously conserve charge, momentum, and energy before time discretization, whereas with any other basis functions the momentum conservation is lost due to lack of translational invariance. 

The electrostatic PIF scheme is studied in 
\cite{evstatiev2013variational,shadwick2014variational,huang2016finite,webb2016spectral}, gyrokinetic PIF scheme in \cite{ohana2016towards} and 
the electromagnetic PIF scheme in 
\cite{shadwick2014variational,ameres2018stochastic,campos2024variational}. However, in all these works it is considered as an aliasing free and highly accurate, albeit computationally not feasible approach, due to the expensive nonuniform discrete Fourier transforms (NUDFT), which are required for the interpolation from the particles to the Fourier space and vice versa. A hybrid gyrokinetic approach with PIF in the poloidal and toroidal directions and PIC in radial direction along with
particle decomposition parallelization (only particles are divided between the MPI ranks and each rank carries all the Fourier modes) is considered in \cite{briguglio2000parallelization} and scalability and efficiency are shown for a small number (up to 10) of processors.

Only in a fairly recent work \cite{mitchell2019efficient}, the authors used nonuniform fast Fourier transforms (NUFFT) \cite{dutt1993fast,dutt1995fast}, in the place of NUDFT, and showed that it is possible to obtain a practical scheme with the same computational complexity as in the standard PIC schemes. Since it was a proof-of-concept study, the authors considered only simple, small-scale examples with shared memory, CPU-based parallelism. Similar to \cite{briguglio2000parallelization} they also remarked that in the case of distributed parallelism a particle decomposition strategy is a sensible approach when the number of Fourier modes required in the simulation is relatively small.

However, when a lot of Fourier modes are required for the simulation then using particle decomposition for spatial parallelization stops scaling after a 
certain number of MPI ranks due to the high serial computation and communication costs associated with the modes. The typical domain decomposition approach of 
spatial parallelization as used in PIC schemes, where both the fields and particles are divided between MPI ranks, is also not very suitable for PIF schemes due to its relatively more global nature. Thus, we need to exploit other ways of parallelization to reduce the time to solution. In this work we propose a strategy for time parallelization of PIF schemes. This, together with the other developments happening in PIF schemes recently, such as \cite{shen2024particle}, where they have been extended to non-periodic boundary
conditions, will make them viable and scalable to very large-scale production simulations. 

Parallel-in-time (PinT) algorithms have a long history starting from the pioneering work of Nievergelt in 1964 \cite{nievergelt1964parallel}. There are different types of PinT algorithms and a comprehensive review can be found in \cite{gander201550,ong2020applications} and more information can be found on the community website~\cite{pintwebsite}. Parareal \cite{lions2001resolution}
is one of the simplest and mostly studied PinT algorithms. However, it has been mostly successful for diffusive problems and the reason for it is the challenge of finding effective coarse propagators in other scenarios.
These need to be cheap and at the same time provide ``accurate enough'' approximations to the fine propagator so that the algorithm converges in a small number of iterations relative to the number of time subdomains. Typically temporal and/or spatial coarsening of the step/mesh sizes is performed to obtain the coarse propagator in parareal. For particle-based algorithms like PIF, which follow the
characteristics, temporal coarsening is possible. However, as we will show later in Section \secref{numerical}, for complicated test cases which involve multiple time scales and oscillatory solutions, temporal coarsening is either not possible or can be done only by a small factor. It is not
clear how to perform spatial coarsening with respect to particles in PIF schemes, as we then need to interpolate between them while they are in random locations.
We can perform spatial coarsening by reducing the number of Fourier modes in the coarse propagator, however, if the high frequency modes are important in the application under consideration, then the coarse propagator will not be a good approximation of the fine propagator. 

In this work, we propose a parareal algorithm for PIF schemes, based on coarse propagators, which are obtained by using PIF with a coarse NUFFT tolerance, or the standard PIC scheme, combined with temporal coarsening. While using PIF with a lenient NUFFT tolerance may seem a natural and obvious algorithmic choice, the use of the standard PIC scheme as a coarse propagator for PIF is a non-standard and novel contribution of this work. This is because PIF is a structure preserving scheme while PIC is not. We support our choice of
coarse propagators through analysis
and validation on test cases which provide further insights into their effectiveness as well as scaling of the error with the number of particles, mesh size, NUFFT tolerance and the time step 
size. 
\emph{The ultimate objective of our work is to make the expensive but stable and highly accurate PIF schemes viable for large scale production level 
electrostatic plasma simulations by using parallelization in time, taking advantage of the modern extreme-scale computing resources.}  


\section{Particle-in-cell method}
\seclab{pic}

    In this work, we consider the non-relativistic \\ Vlasov-Poisson system with a fixed magnetic field and introduce the PIC method in that setting. Since PIC is more familiar and shares many features with PIF we consider it first. This will be also useful later since it will be used as one of the coarse propagators in the parareal algorithm. 

    The electrons are immersed in a uniform, immobile, neutralizing background ion population, and the electron dynamics is given by
    \begin{equation}
        \eqnlab{Vlasov}
        \frac{\partial f}{\partial t} + \vb \cdot \nabla_{\xb} f + \frac{q_e}{m_e}\LRp{\Eb + \vb \times \B_{ext}} \cdot \nabla_{\vb} f = 0,
    \end{equation}
where $\Eb = \Eb_{sc} + \Eb_{ext}$ is the total electric field, $\Eb_{ext}$ and $\B_{ext}$ are the known external electric and magnetic fields. Here, $f(\xb,\vb,t)$ is the electron phase-space distribution and $q_e$, $m_e$ are the electron charge and mass respectively. The total
electron charge in the system is given by $Q_e=q_e\int\int f d\xb d\vb$, the electron charge density by $\rho_e(\xb) = q_e\int f d\vb$ and the
constant ion density by $\rho_i = Q_e/\int d\xb$. Let us denote the permittivity of free space by $\varepsilon_0$. The self-consistent potential ($\phi$) and the electric field ($\Eb_{sc}$) due to space charge are given by 
\begin{align}
    \eqnlab{potential}
    \quad -\Delta \phi &= \rho/\varepsilon_0 = \LRp{\rho_e - \rho_i}/\varepsilon_0, \\
    \Eb_{sc} &= -\nabla \phi.
\end{align}

The PIC method discretizes the phase-space distribution $f(\xb,\vb,t)$ in a Lagrangian way by means of macro-particles (hereafter referred to as ``particles'' for simplicity) as 
\begin{equation}
 \eqnlab{dist_f}
     f\LRp{\xb,\vb,t} = \sum_{j=1}^{N_p} w_j S\LRp{\xb-\xb_j} \delta\LRp{\vb-\vb_j},
 \end{equation}
 where $w_j$ is the particle weight, $S\LRp{\xb-\xb_j}$ is the shape function in the configuration space and $\delta\LRp{\vb-\vb_j}$ is the Dirac-delta function in the velocity space. At time $t=0$, the
distribution $f$ is sampled, which leads to the creation of computational particles. Subsequently, a typical computational cycle in PIC with fast Fourier transform (FFT)-based field solver\footnote{Even though there are other flavors in PIC we choose this particular type as it is the simplest and the most relevant type for the current study.} consists of the following steps:

\begin{enumerate}
    \item \label{step1_pic} Assign a shape function - e.g. cloud-in-cell \cite{birdsall2004plasma} - to each particle $j$ and deposit the electron charge onto an underlying mesh. This is known as ``scatter'' in PIC.
    \item \label{step2_pic} Take FFT of the charge density and solve the Poisson equation in Fourier space. Use inverse FFT to get the electric field on the grid.
    \item Interpolate $\Eb$ from the grid points to the particle locations ${\xb}_j$ using the same interpolation function as in the scatter operation. This is 
        typically known as ``gather'' in PIC.
    \item By means of a time integrator, advance the particle positions and velocities using
        \begin{equation}
        \eqnlab{velposPIF}
            \frac{d\vb_j}{dt} = \frac{q_e}{m_e}\LRp{\Eb + \vb \times \B_{ext}}|_{\xb=\xb_j}, \quad
            \frac{d\xb_j}{dt} = \vb_j.
        \end{equation}
\end{enumerate}

\section{Particle-in-Fourier method}
\seclab{pif}

 The PIF scheme, similar to the PIC scheme, discretizes the distribution function $f$, by means of particles as in \eqnref{dist_f}.
 The main difference, however, is that the charge density is scattered directly onto the Fourier space which gives
 \begin{align} 
\rhok\LRp{\kb} &= \frac{1}{L^3}\int\rho\LRp{\xb}\exp\LRp{-i\kb\cdot\xb} d\xb, \\
\eqnlab{scatter_pif}
            &= \frac{q_e\Sk}{L^3}\sum_{j=1}^{N_p}\exp\LRp{{-i\kb\cdot\xb_j}}.
\end{align}
This in contrast to PIC, in steps \ref{step1_pic} and \ref{step2_pic}, where the charge density is first scattered onto a real space grid, and then a uniform FFT is used to transform it to  the Fourier space.

For the simplicity of exposition we consider three dimensions, $N$ is the number of Fourier modes in each dimension, $N_m = N^3$ is the total number 
of modes and $L$ is the length of the domain in each dimension. We 
also assume $N$ to be even and take the Fourier modes $\kb \in K_N = \LRc{\frac{2\pi}{L}[0,N-1]}^3$. $\Sk=\mc{F}\LRp{S(x)}$ is the Fourier transform of the shape function $S$ which is usually available in analytic form. With the charge density in the Fourier space the Poisson equation is solved. Then the electric field is gathered from the Fourier space to the particle positions using
 \begin{equation}
\eqnlab{gather_pif}
\Eb(\xb_j) = \sum_{\kb\in K_N}\Ebk\Sk\exp\LRp{{i\kb\cdot\xb_j}}.
\end{equation}
With this electric field the particle velocities and positions are updated using equation \eqnref{velposPIF} in the same way as for the PIC schemes.

Naive calculations of the interpolations from particles to Fourier modes in equation \eqnref{scatter_pif} and vice versa in equation \eqnref{gather_pif} using discrete Fourier transforms (DFT) cost $\bigO\LRp{N_pN_m}$ which are prohibitively expensive except when only a small number of particles and Fourier modes are considered. Hence, the key step as introduced in \cite{mitchell2019efficient} that makes PIF schemes practical is to NUFFTs of type 1 and 2 \cite{dutt1993fast,dutt1995fast,potts2001fast,barnett2019parallel} to compute equations
\eqnref{scatter_pif} and \eqnref{gather_pif}, respectively. This reduces the complexity of PIF schemes to $\mc{O}\LRp{\LRp{|log\varepsilon|+1}^dN_p+N_mlogN_m}$ comparable to PIC 
schemes, although with a much bigger constant in front of $N_p$. The constant depends on the tolerance $\varepsilon$ chosen for the NUFFT as unlike uniform FFT NUFFT is an approximate algorithm.

Compared to PIC schemes, PIF schemes with NUFFT are relatively more global in nature. 
Therefore, particle 
decomposition, where only the particles are split between the MPI ranks, and not the modes, is suggested for spatial parallelization in \cite{briguglio2000parallelization,mitchell2019efficient}. However, it poses a bottleneck in scaling when a large number of Fourier modes are required for the simulation.

When scaling in the spatial direction stops, in order to further decrease the
time to solution and to scale the simulation, we can try to exploit parallelism in the time integration. In the following 
we describe the parareal algorithm and perform time parallelization of PIF schemes with NUDFT or NUFFT based on this idea. 

\section{Parareal for PIF}
\seclab{parareal}

The most widely investigated PinT algorithm, parareal \cite{lions2001resolution}, is an iterative scheme that can be thought of
as a predictor-corrector approach. The initial time domain is split into multiple intervals and a cheap (possibly inaccurate) predictor called ``coarse propagator'' ($G$) is run serially to give initial guesses for an accurate and expensive corrector called ``fine propagator'' ($F$). With these
guesses, the fine propagator is run in parallel in each of these intervals. The parareal correction step is performed at the end of each iteration as follows
\begin{align}
    U_0^{k+1} &:= u ^0, \\
\eqnlab{parareal_correction}
    U_{n+1}^{k+1} &:=
    F(T_{n+1}, T_n, U_n^k)
    + G(T_{n+1}, T_n, U_n^{k+1})
    - G(T_{n+1}, T_n, U_n^{k}),
\end{align}
where, $[T_n,T_{n+1}]$ are the time intervals, $\Delta T$ is the size of the time subdomain, $U_n^{k}$ is the approximation of the solution at $T_n$ in the $k^{\text{th}}$ parareal iteration, and $u^0$ is the initial condition. The iterations are performed until convergence to a specific tolerance. In order for the
parareal algorithm to give practical speedups over the sequential time integration, the following two criteria must be met:
\begin{itemize}
     \item The cost of the coarse propagator together with the communications costs is much less than that of the fine propagator.
     \item The number of parareal iterations needed for convergence is much smaller than the number of time subdomains. 
\end{itemize}
Typically, the coarse propagators for parareal are created by coarsening the time step size and/or the spatial mesh size. 
However, as mentioned in the Introduction, this is either not possible or will not be effective in the context of PIF schemes. 

Instead, given a PIF scheme with NUFFT of tolerance $\varepsilon_f$ as a fine propagator we use either PIF scheme with a coarse NUFFT tolerance $\varepsilon_g > \varepsilon_f$ or the standard PIC scheme as the 
coarse propagator in the parareal algorithm. Both PIF and PIC coarse propagators may or may not employ time coarsening depending on the fine time step size and the stability of the time integrators. 
\emph{The parareal correction equation \eqnref{parareal_correction} is performed in the phase-space on both the positions and velocities of the particles, i.e., on the vector $\ub=\LRc{\xb,\vb}$.} In the next section we state and prove error bounds for the parareal algorithm with PIC and PIF as coarse propagators.   

\section{Theoretical error analysis}
\seclab{theory_convergence}
\subsection{Matrix-vector formulation of PIF}
Let us denote by $\Pb$ the linear map (from particles to Fourier space) for the exact PIF scheme\footnote{Exact here refers to the use of NUDFT instead of NUFFT. However, the PIF scheme still only uses a finite number of Fourier modes and particles.} with NUDFT which interpolates the charge $q_e$ 
from the particles to the density in the Fourier space $\rhok$. The size of the matrix $\Pb$ is $N_m \times N_p$ and its entries are
\begin{equation}
\eqnlab{ppif}
    \Pb_{lj} = \sum_{p=0}^{N_m-1}\frac{1}{\sqrt{L^3}}\LRp{\Sbk}_{lp}\exp\LRp{{-i\kb_p\cdot\xb_j}}, 
\end{equation}
where $\Sbk$ is a diagonal shape function matrix of size $N_m \times N_m$ in the Fourier space. We denote by $\PbH$ the conjugate transpose of the matrix $\Pb$ and this adjoint matrix interpolates quantities from the Fourier space to the particle locations. 

Now with the charge density in the Fourier space $\rhok$ the electric field is given by $\Ebk = -\frac{i\kb}{|\kb|^2} \rhok$.
If we denote by $\Lb$ a diagonal matrix of size $N_m \times N_m$ with entries $\Lb_{jj} = -\frac{i\kb_j}{|\kb_j|^2}$ then 
equation \eqnref{velposPIF} can be written as
\begin{align}
\eqnlab{posPIFmatrix}
\DD{\xb}{t} &= \vb, \\
\eqnlab{velPIFmatrix}
    \DD{\vb}{t} &= \frac{q_e}{m_e}\LRp{\PbH\Lb\Pb\qb_e + \vb\times\Bb_{ext}}, 
\end{align}
where $\xb,\vb,$ denote the vectors with the positions and velocities of all the particles and $\qb_e = q_e\cdot\mathbf{1}$ is a vector of size $N_p$. 

\subsection{Bound for a coarse propagator}
We construct coarse propagators by approximating $\Pb$ and $\PbH$ by $\Pbh$ and $\PbhH$. 
Since the equations for the coarse propagator differ only by the term $\PbhH\Lb\Pbh$ instead of $\PbH\Lb\Pb$ in equations \eqnref{posPIFmatrix} and \eqnref{velPIFmatrix}, we bound $\norm{\PbH\Lb\Pb-\PbhH\Lb\Pbh}$ in the following lemma. Since the analysis works in any norm we do not specify a particular one throughout
the theoretical study. In the numerical results in Section \secref{numerical} we use the $L^2$ norm in phase-space and $L^{\infty}$ norm 
in time. 

\begin{lemma}
    \lemlab{rhs_bound}
    The difference between right hand side term of the exact NUDFT PIF scheme $\PbH\Lb\Pb$ and the coarse propagator $\PbhH\Lb\Pbh$ satisfies
    \begin{equation}
        \eqnlab{rhs_bound}
        \norm{\PbH\Lb\Pb-\PbhH\Lb\Pbh} \le \norm{\Lb} \norm{\Pb-\Pbh} \LRp{\norm{\Pb} + \norm{\Pbh}}.
    \end{equation}
\end{lemma}

\begin{proof}
    \begin{align*}
        \norm{\PbH\Lb\Pb - \PbhH\Lb\Pbh} 
        &\le \norm{\PbH\Lb\LRp{\Pb - \Pbh}} + \norm{\LRp{\PbH - \PbhH}\Lb\Pbh},\\
        &\le \norm{\PbH}\norm{\Lb}\norm{\Pb - \Pbh} + \norm{\PbH - \PbhH}\norm{\Lb}\norm{\Pbh},\\
        &\le \norm{\Lb} \norm{\Pb-\Pbh} \LRp{\norm{\Pb} + \norm{\Pbh}}.
    \end{align*}
    In the above proof we have used the triangle inequality as well as the equality of norms of a matrix and its adjoint (conjugate transpose). 
\end{proof}

\subsection{PIC and NUFFT PIF as coarse propagators}
If we consider the PIC scheme with FFT-based field solver as the coarse propagator then the entries of matrix $\Pbh$ are given by
\begin{equation}
\eqnlab{ppic}
    \LRp{\Pbh}_{lj} = \sum_{p=0}^{N_g-1}\frac{1}{\sqrt{N^3}}\exp\LRp{{-i\tilde{\kb}_l\cdot\nb_p}}\Sb_{pj}, 
\end{equation}
where for simplicity we take $N$ number of grid points in $x$, $y$ and $z$ directions, $N_g=N^3$ is the total number of grid points, $\tilde{\kb} = \LRc{\frac{2\pi}{N}\LRs{0,N-1}}^3$. Also, $\nb=\LRc{\LRs{0,N-1}}^3$ and $\Sb_{pj}$ is a shape function matrix of size $N_g \times N_p$ in the real space and interpolates charges of the particles to charge density $\rho$ in the real space. In equation \eqnref{ppic} we have used the unitary representation of the Fourier transform. Comparing equations \eqnref{ppic} and \eqnref{ppif} we make the following observation: the PIC scheme first interpolates the charges onto a grid using shape function $\Sb$ in real space and then performs a 
uniform Fourier transform of the charge density to get the charge density in the 
Fourier space $\rhok$. On the other hand the PIF scheme directly 
interpolates the charges onto the Fourier space using non-uniform 
Fourier transform and then multiples it with the analytic Fourier 
transform of the shape function.

Next, if we use the PIF scheme with NUFFT of tolerance $\varepsilon$ instead of NUDFT as the coarse propagator, then the matrix $\Pbh$ is given by the NUFFT matrix of type 1
and $\PbhH$ by the NUFFT matrix of type 2 \cite{barnett2019parallel}. In addition, with either of these coarse propagators (PIC or NUFFT PIF with tolerance $\varepsilon$) we also use a time discretization of 
local truncation error $p+1$ as opposed to the exact time integration in equations \eqnref{posPIFmatrix} and \eqnref{velPIFmatrix}. Let us denote by $\ub=\LRc{\xb,\vb}$ the solution obtained 
by solving PIF equations \eqnref{posPIFmatrix} and \eqnref{velPIFmatrix} with exact time integration and by $\eb_{n+1}^k=\LRc{\norm{\xb_{n+1}^k - \xb_{n+1}},\norm{\vb_{n+1}^k - \vb_{n+1}}}$ the error at the $k^{th}$ iteration of parareal on time subdomain $\Omega_n$ with respect to the exact PIF solution $\ub_{n+1}$. Now we state and prove the error bounds for the parareal algorithm with PIC or NUFFT PIF as coarse propagators in the following theorems.  

\subsection{Error bound on parareal with PIC as coarse propagator}
\begin{theorem}
\theolab{conv_pic}
    Let $\Fb\LRp{\ub_{n}^k}$ be the NUDFT PIF solution of equations \eqnref{posPIFmatrix}, \eqnref{velPIFmatrix} obtained with the exact time integrator on the time subdomain $\Omega_{n}$, and let $\Gb\LRp{\ub_{n}^k}$ be the approximate solution obtained with the PIC scheme with $\Pbh$ given by equation \eqnref{ppic} and a time integrator of step size $\Delta t_g$ such that $\Delta T=n_g\Delta t_g$ and local truncation error bounded by $\mc{O}\LRp{\Delta t_g^{p+1}}$. Let a B-spline shape
    function of order $m$ be used in equation \eqnref{ppic} and its analytical Fourier transform in equation \eqnref{ppif}. Let $N_p$ be the total number of particles randomly sampled from an initial distribution function, $N_m=N_g$ the total number of modes and grid points in PIF and PIC schemes respectively, $h$ the mesh size and $P_c=N_p/N_g$ the number of particles per mode or cell for the fine and coarse propagators. Assuming that the coarse propagator has a Lipschitz constant of $C_{pic}$, then at iteration $k$ of the parareal algorithm, we have the following bound:
    \begin{equation}
        \eqnlab{conv_pic}
        \eb_{n+1}^k \le \bar{C}^{n-k} \frac{\LRp{C_{grid} h^{\min\LRp{m+1,2}} + C_{noise} P_c^{-0.5} + C_{time} \Delta T \Delta t_g^p}^k}{k!}\prod_{j=1}^{k}(n+1-j) \deltab,  
    \end{equation}
        where $\deltab = \max\limits_{n=1,\ldots,N}\eb_n^0$ and $\eb_n^0$ is the initial error. Here, $\bar{C}=\max(1,C_{pic})$, and $C_{grid}$, $C_{noise}$ are constants related to the density distribution and the shape function used in the PIC scheme. $C_{time}$ is a constant which depends on the time integrator and the smoothness of the distribution with respect to time. 
\end{theorem}

\begin{proof}
    We follow \cite{gander2008nonlinear,gander2023unified} where the authors show that if $\alpha=\norm{\Fb-\Gb}$ and $\beta=\norm{\Gb}$ then the convergence of the parareal algorithm for general nonlinear ODEs is given by
    \begin{equation}
        \eqnlab{gander_parareal}
    \eb_{n+1}^k \le \bar{\beta}^{n-k} \frac{\alpha^k}{k!}\prod_{j=1}^{k}(n+1-j) \deltab  
    \end{equation}
    where $\bar{\beta}=\max\LRp{1,\beta}$. Since we have assumed a Lipschitz constant $C_{pic}$ for the coarse propagator we only need to evaluate $\norm{\Fb-\Gb}$ and then the proof follows immediately from equation \eqnref{gander_parareal}.
    If we denote by $\LRp{\Fb-\Gb}_{\xb,\vb}$ and $\LRp{\Fb-\Gb}_{t}$ the differences between $\Fb$ and $\Gb$ in the spatial (phase-space) and temporal dimensions, then  
    \[
        \norm{\Fb-\Gb} = \norm{\LRp{\Fb-\Gb}_{\xb,\vb}+\LRp{\Fb-\Gb}_{t}} \le \norm{\LRp{\Fb-\Gb}_{\xb,\vb}}+\norm{\LRp{\Fb-\Gb}_{t}}.
    \]
    As mentioned before, in the spatial approximation, the fine and coarse propagators differ by $\norm{\PbH\Lb\Pb-\PbhH\Lb\Pbh}$ and from Lemma \lemref{rhs_bound} we can write 
    \begin{align}
        \norm{\LRp{\Fb-\Gb}_{\xb,\vb}} &= \norm{\PbH\Lb\Pb-\PbhH\Lb\Pbh},\\
        &\le \norm{\Lb} \norm{\Pb-\Pbh} \LRp{\norm{\Pb} + \norm{\Pbh}},\\
        \eqnlab{space_bound_1}
        &\le \gamma \norm{\Pb-\Pbh},
    \end{align}
    where $\gamma$ is a constant which can be chosen independently of the grid resolution and the total number of particles. Equation \eqnref{space_bound_1} follows from the fact that the norms $\norm{\Lb}$, $\norm{\Pb}$ and $\norm{\Pbh}$ can be bounded independent of the grid size and the total number of particles.

    Now the difference between the ideal linear map and the approximate linear map from PIC has two components of error: the grid-based error and the statistical noise due to the particles. If we use a B-spline shape function of order $m$ then it has a grid-based error of $\mc{O}(h^{\min\LRp{m+1,2}})$ where $h$ is the mesh size \cite{birdsall2004plasma,ricketson2016sparse,muralikrishnan2021sparse}. The error due to the particle noise scales as $\mc{O}(P_c^{-0.5})$ where $P_c = N_p/N_g$ is the
    number of particles per cell in the PIC scheme \cite{birdsall2004plasma,ricketson2016sparse,muralikrishnan2021sparse}. Plugging these two error terms in equation \eqnref{space_bound_1} we get
    \begin{align}
        \norm{\LRp{\Fb-\Gb}_{\xb,\vb}} &\le \gamma \LRp{\tilde{C}_{grid} h^{\min\LRp{m+1,2}} + \tilde{C}_{noise} P_c^{-0.5}},\\
        \eqnlab{space_bound_2}
        &\le C_{grid} h^{\min\LRp{m+1,2}} + C_{noise} P_c^{-0.5},
    \end{align}
    where $\tilde{C}_{grid}$ is a constant related to the grid-based error and it depends on the norms of the configuration space ($\xb$) derivatives of the density and the order of the B-spline function \cite{ricketson2016sparse,muralikrishnan2021sparse}. $\tilde{C}_{noise}$ is a constant related to the statistical noise which in turn depends on the order of the B-spline function and the norm of the density \cite{ricketson2016sparse,muralikrishnan2021sparse}. The constants $C_{grid} = \gamma\tilde{C}_{grid}$ and $C_{noise} = \gamma\tilde{C}_{noise}$.

    Now since the time integrator has a local truncation error of $\mc{O}(\Delta t_g^{p+1})$ and we perform $n_g=\Delta T/\Delta t_g$ steps of it to reach the end of the time subdomain $\Omega_n$, we can write $\norm{\LRp{\Fb-\Gb}_{t}} = C_{time} \Delta T \Delta t_g^p$ where $C_{time}$ is a constant which depends on the norms of the derivatives of the distribution with respect to time and the choice of the time integrator. Combining the spatial and temporal approximation bounds, we get

    \[
        \norm{\Fb-\Gb} \le C_{grid} h^{\min\LRp{m+1,2}} + C_{noise} P_c^{-0.5} + C_{time} \Delta T \Delta t_g^p,
    \]
    which together with the Lipschitz constant $C_{pic}$ of the coarse propagator yields the desired result when plugged into equation \eqnref{gander_parareal}.
\end{proof}

\subsection{Error bound on parareal with approximate PIF as coarse propagator}
\begin{theorem}
\theolab{conv_pif}
    Let $\Fb\LRp{\ub_{n}^k}$ be the NUDFT PIF solution of equations \eqnref{posPIFmatrix}, \eqnref{velPIFmatrix} obtained with the exact time integrator on the time subdomain $\Omega_{n}$, and let $\Gb\LRp{\ub_{n}^k}$ be the approximate solution obtained with the PIF scheme using NUFFT of tolerance $\varepsilon$ and a time integrator of step size $\Delta t_g$  such that $\Delta T=n_g\Delta t_g$ and the local truncation error bounded by $\mc{O}\LRp{\Delta t_g^{p+1}}$. Let the total number of
    particles $N_p$, modes $N_m$ and the shape function be same for both the fine and coarse propagators. Assuming that the coarse propagator has a Lipschitz constant of $C_{pif}$, then at iteration $k$ of the parareal algorithm, we have the following bound:
    \begin{equation}
        \eqnlab{conv_pif}
        \eb_{n+1}^k \le \bar{C}^{n-k} \frac{\LRp{C_{nufft} \varepsilon + C_{time} \Delta T \Delta t_g^p}^k}{k!}\prod_{j=1}^{k}(n+1-j) \deltab  
    \end{equation}
        where $\deltab = \max\limits_{n=1,\ldots,N}\eb_n^0$ and $\eb_n^0$ is the initial error. Here, $\bar{C}=\max(1,C_{pif})$, and $C_{nufft}$ is a constant related to the NUFFT scheme. $C_{time}$ is a constant which depends on the time integrator and the smoothness of the distribution with respect to time. 
\end{theorem}

\begin{proof}
    The proof is very similar to that of Theorem \theoref{conv_pic}. The difference comes from the term $\norm{\Pb-\Pbh}$ which when we use a NUFFT of tolerance $\varepsilon$ is $\mc{O}(\varepsilon)$ \cite{barnett2019parallel,barnett2021aliasing}. Combining this with the bound on the temporal error we get   
    \[
        \norm{\Fb-\Gb} \le C_{nufft} \varepsilon + C_{time} \Delta T t_g^p,
    \]
    where $C_{nufft}$ is a constant which depends on the NUFFT scheme used and $C_{time}$ depends on the norms of the derivatives of the distribution with respect to time and the choice of the time integrator. Putting this together with the Lipschitz constant $C_{pif}$ of the coarse propagator in equation \eqnref{gander_parareal} ends the proof.  
\end{proof}

\begin{remark}
PIF with NUDFT or NUFFT with very small tolerances close to machine precision is structure preserving. When we choose the tolerance of the parareal algorithm also close to machine precision then the solution obtained after convergence is structure preserving. However, if a lenient tolerance is used for the parareal then the solution may not have the same structure preserving properties. We will investigate symmetric parareal algorithms in \cite{dai2013symmetric} as part of our future work since it preserves the structure preserving properties in each iteration.  
\end{remark}


\section{Numerical results}
\seclab{numerical}
\subsection{Mini-apps}
We consider the mini-apps with the parameters described in \cite{muralikrishnan2024scaling} for the numerical verification of the theoretical results in Section \secref{theory_convergence} as well as for the parallel scaling study. We briefly describe the test cases and their parameters here for the sake of completeness. 
The parareal algorithm and the mini-apps are implemented in the performance portable, open-source, C++ library
IPPL \cite{muralikrishnan2024scaling,matthias_frey_2024_10878166} 
interfaced with the FINUFFT library \cite{barnett2019parallel,shih2021cufinufft} for performing the NUFFTs in an efficient manner. All 
the computations are performed on NVIDIA A100 GPUs in the JUWELS Booster supercomputer at the Jülich Supercomputing Centre.
\subsubsection{Landau damping}
This is one of the classical benchmark problems in plasma physics. We consider the following initial distribution
   \begin{equation}
        f(t=0) = \frac{1}{\LRp{2\pi}^{3/2}} e^{-|\vb|^2/2} \LRp{1+\alpha\cos(w x)} \LRp{1+\alpha\cos(w y)} 
               \LRp{1+\alpha\cos(w z)}
    \end{equation}
in the domain $\LRs{0,L}^3$, where $L = 2\pi/w$ is the length in each dimension. We choose the 
following parameters: $w = 0.5$, $\alpha=0.05$ which correspond to weak Landau damping. The total electron charge based 
on our initial distribution is $Q_e = -L^3$.

\subsubsection{Two-stream instability}
Similar to Landau damping, this is another classical benchmark problem in plasma physics. We consider the following initial distribution of electrons

    \begin{equation}
        f(t=0) = \frac{1}{\sigma^3\LRp{2\pi}^{3/2}}\LRc{0.5e^{-\frac{|\vb-\vb_{b1}|^2}{2\sigma^2}} + 0.5e^{-\frac{|\vb-\vb_{b2}|^2}{2\sigma^2}}} \LRp{1+\alpha\cos(wz)}
    \end{equation}
    in the domain $\LRs{0,L}^3$, where $L = 2\pi/w$ is the length in each dimension. We choose $\sigma=0.1$, $w=0.5$, 
    $\alpha=0.01$, $\vb_{b1}=\LRc{0,0,-\pi/2}$, and $\vb_{b2}=\LRc{0,0,\pi/2}$. The total charge $Q_e$ is chosen in the same 
    way as in the Landau damping example.

\subsubsection{Penning trap}
This mini-app corresponds to the dynamics of electrons in a Penning trap with a neutralizing static ion background. Unlike the Landau damping and two-stream instability test cases this one involves external electric and magnetic fields. The external magnetic field is given by $\B_{ext}=\LRc{0,0,5}$ and the quadrupole external electric field by 
    \begin{equation}
        \eqnlab{penning_ext_efield}
        \Eb_{ext} = \LRp{-\frac{15}{L}\LRp{x-\frac{L}{2}},-\frac{15}{L}\LRp{y-\frac{L}{2}},\frac{30}{L}\LRp{z-\frac{L}{2}}},
    \end{equation}
where the domain is $\LRs{0,L}^3$ and $L=25$. For the initial conditions, we sample the phase-space using a Gaussian
    distribution in all the variables. The mean and standard deviation for
    all the velocity components are $0$ and $1$, respectively. While the mean
    for all the configuration space variables is $L/2$, the standard 
    deviations are $2$, $1$ and $3$ for $x$, $y$, and $z$, respectively. The total electron charge is 
    $Q_e=-1562.5$. 

\subsection{Verification of theoretical estimates}
We verify the theoretical error scalings in Section \secref{theory_convergence} using the mini-apps described in the previous section. The particles are randomly sampled from the initial distribution functions by the inverse transform sampling 
technique as in \cite{muralikrishnan2024scaling}. For the time integration in both the coarse and fine propagators, we use the kick-drift-kick form of 
the velocity Verlet or Boris scheme as in \cite{tretiak2019arbitrary}, and this gives both the positions and velocities of the particles at 
integer time steps. The final time $T=19.2$ is chosen for all tests
except for the bottom row of Figure \figref{conv_dt_pic_penning} in which we take $T=1.2$ for reasons that will be explained later. We use 16 time subdomains or GPUs to parallelize in the time direction whereas in the spatial 
direction either 16 GPUs or 4 GPUs are used depending on the problem size. A linear B-spline (cloud-in-cell) shape function
is used for the PIC scheme and the analytical Fourier transform of it is used for the PIF scheme. 
In order to reduce the computational costs, instead of PIF with NUDFT as a fine propagator 
we choose NUFFT PIF with a tight tolerance of $10^{-12}$ for Figures \figref{conv_pc}, \figref{conv_h}, \figref{conv_dt_pic_penning}, \figref{conv_dt_pic_pif_landau} (top row) and  \figref{conv_epsilon}. For the verification of convergence with respect to the coarse time step size in the bottom row of Figure \figref{conv_dt_pic_pif_landau}
we choose the fine and coarse NUFFT tolerances as $10^{-6}$ to eliminate the spatial component of error. Similarly, to eliminate the temporal component of error for Figures \figref{conv_pc}, \figref{conv_h}, and \figref{conv_epsilon},
we select the same time step size of $0.05$ for both the fine and the coarse propagators. For parareal we choose the stopping criterion based on 
\begin{equation}
    \eqnlab{stop_criteria}
    \frac{\norm{\Gb\LRp{\xb_{n}^{k+1}} - \Gb\LRp{\xb_{n}^{k}}}_2}{\norm{\Gb\LRp{\xb_{n}^{k+1}}}_2} \leq \epsilon \quad \text{and} \quad \frac{\norm{\Gb\LRp{\vb_{n}^{k+1}} - \Gb\LRp{\vb_{n}^{k}}}_2}{\norm{\Gb\LRp{\vb_{n}^{k+1}}}_2} \leq \epsilon
\end{equation}
for time subdomain $\Omega_{n}$, where $\epsilon=10^{-11}$ is chosen as the stopping tolerance. Finally, the $L^{\infty}$ norm of the error across the time subdomains is shown in the figures. 
Now this is not the same as $\eb_{n+1}^k$ in Theorems \theoref{conv_pic} and \theoref{conv_pif} where the error is between the exact solution (or serial fine propagator) and the parareal solution. However, the stopping 
criterion \eqnref{stop_criteria} is one of the practical criteria used in the simulations and since the error $\eb_{n+1}^k$ differs from it only by an additional Lipschitz constant of the fine propagator
we can still expect it to follow the same scalings as in Theorems \theoref{conv_pic} and \theoref{conv_pif}.

\begin{figure}[h!b!t!]
    \subfigure[Error Vs Iterations]{
    \includegraphics[width=0.48\columnwidth]{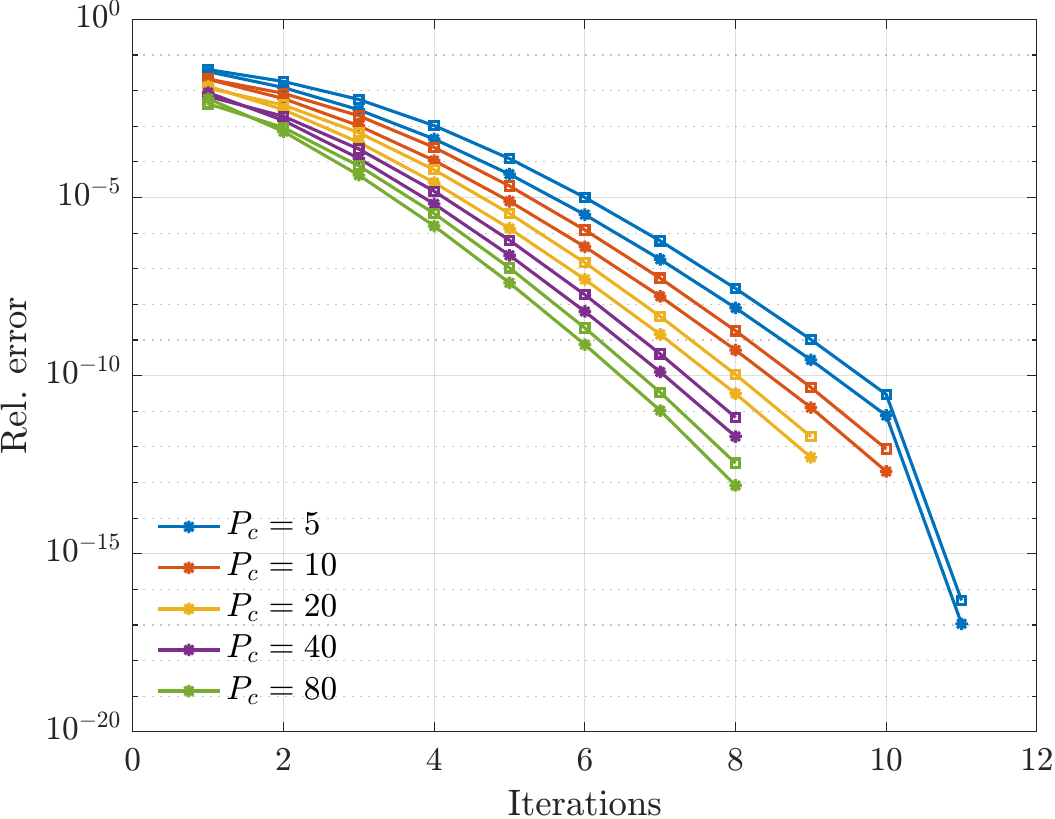}
  }
  \subfigure[Error Vs $P_c$]{
    \includegraphics[width=0.46\columnwidth]{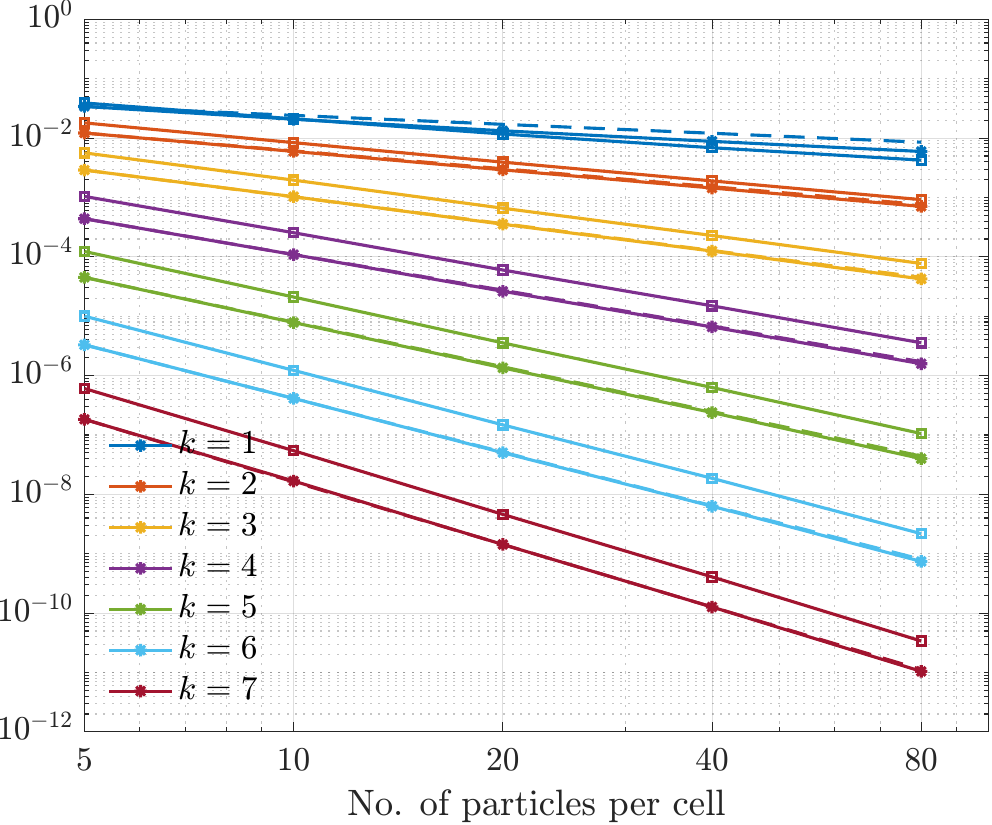}
    \figlab{conv_pc_right}
  }
    \caption{Landau damping: Relative error versus the number of particles per cell $P_c$ and iterations $k$ with PIC as coarse propagator. The parameters are $N_m=32^3$, $T=19.2$ and $\Delta t_f = \Delta t_g = 0.05$. The fine PIF propagator has the same number of particles per mode as the number of particles per cell in the coarse propagator. The dashed lines in the right figure represent the theoretical scaling $\mc{O}\LRp{P_c^{-0.5k}}$ obtained from Theorem \theoref{conv_pic}. The square
    and asterisk markers represent the relative errors in velocity $\vb$ and position $\xb$ respectively.}   
\figlab{conv_pc}
\end{figure}

For verification, we take one component of error at a time and select the parameters such that the other error components are either zero or minimized. For example, while studying the effect of the number of particles per cell on the convergence of the algorithm we take the coarse and fine time step sizes to be same so that time error component is zero. Also, we  use a sufficiently fine grid so that the grid-based error is small.
We show representative results from one of the mini-apps and comment on the other cases for brevity. 

\subsubsection{Convergence with particles per cell}
First we consider PIC as coarse propagator and verify the error scaling with respect to 
the particles per cell $P_c$ in Figure \figref{conv_pc}. We consider the Landau damping mini-app with $32^3$ modes in PIF for the fine propagator and the same number of grid cells in PIC for the coarse propagator. Since the error is dominated by
statistical noise in this test case, the number of modes/grid cells is sufficient to minimize the grid-based error term in equation \eqnref{conv_pic}. The temporal component of the error is zero, since we select the same time step size of $0.05$ for both the coarse and the fine propagators. 

We plot the decrease in error with iterations $k$ for different numbers of particles per cell $P_c$ and the error versus $P_c$ for different $k$ in the left and right columns of Figure \figref{conv_pc}. In general, it is difficult to evaluate the different constants in Theorems \theoref{conv_pic} and \theoref{conv_pif} from the numerical results. This is the reason we verify the theoretical scalings only in the right columns of all the figures, since the constants only affect the intercepts but
not the slopes in them. However, with ad-hoc constants we were able to verify the superlinear convergence of the error with respect to iterations as given by Theorems \theoref{conv_pic} and \theoref{conv_pif} for all the left columns of figures. From Figure \figref{conv_pc_right} we see good agreement of the numerical results with the theoretical scaling. We observed similar convergence results for the two-stream instability as well as the Penning trap test cases. However, since the
initial distribution is Gaussian in the case of the Penning trap, the total error is not dominated by statistical noise for the values of particles per cell tested in Figure \figref{conv_pc}. Hence, we needed to go to lower numbers of total particles in order to observe convergence there. We note here that the particles are randomly sampled and we do not employ any noise reduction techniques such as quasi Monte Carlo methods as Theorem \theoref{conv_pic} does not cover them. But, in practice,  such techniques can help to achieve
faster convergence in the parareal algorithm if the total error is dominated by noise in the test case.

\begin{figure}[h!b!t!]
    \subfigure[Error Vs Iterations]{
    \includegraphics[width=0.48\columnwidth]{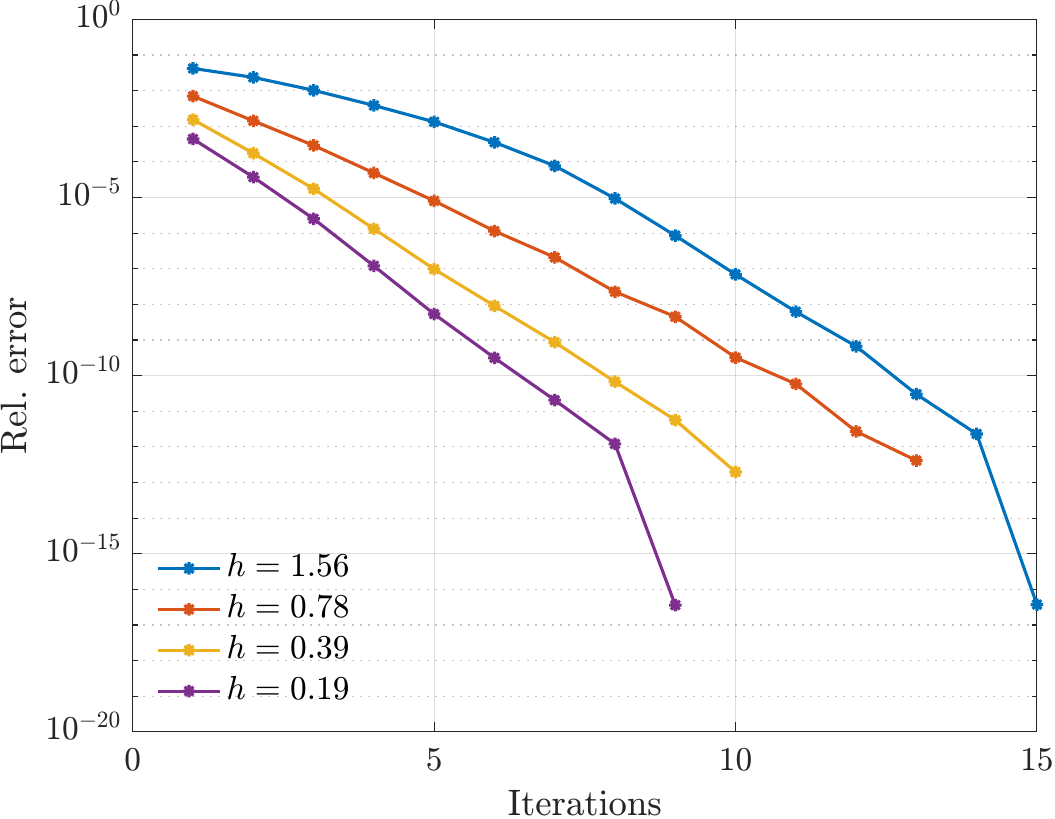}
  }
  \subfigure[Error Vs $h$]{
    \includegraphics[width=0.46\columnwidth]{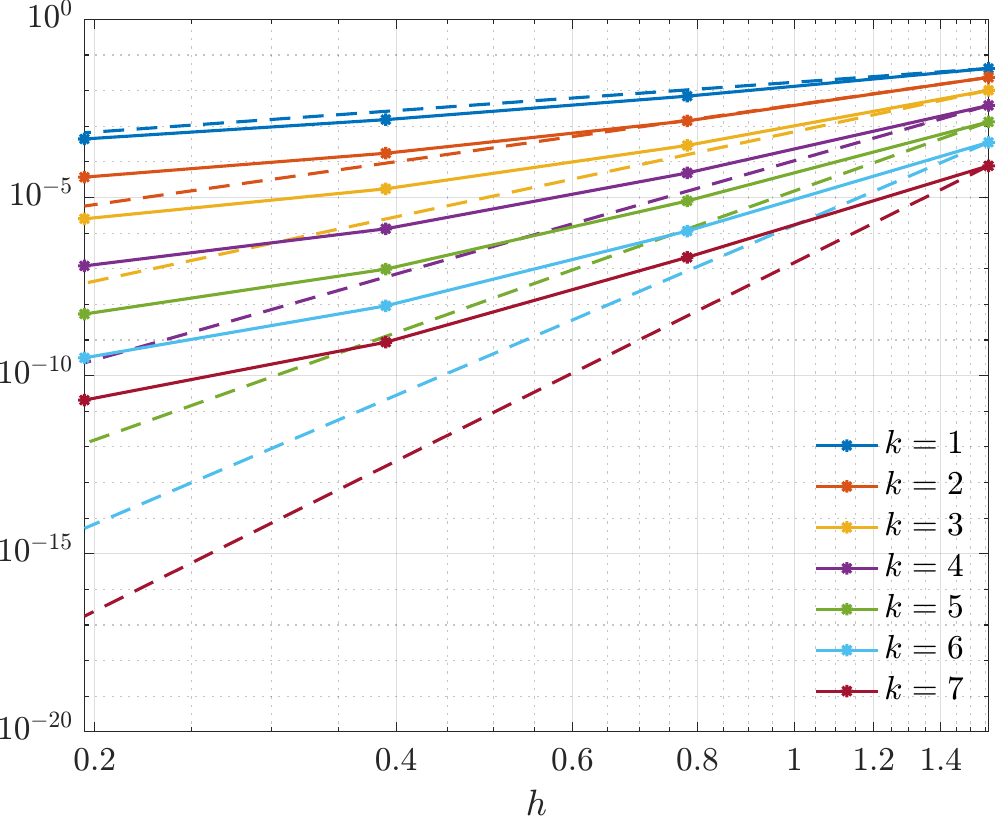}
    \figlab{conv_h_right}
  }
    \caption{Penning trap: Relative error in position $\xb$ versus the mesh size $h$ and iterations $k$ with PIC as the coarse propagator. The parameters are $P_c=10$, $T=19.2$ and $\Delta t_f = \Delta t_g = 0.05$. The fine PIF propagator has $16^3$, $32^3$, $64^3$ and $128^3$ modes corresponding to the mesh sizes in the coarse propagator. The dashed lines in the right figure represent the theoretical scaling $\mc{O}\LRp{h^{2k}}$ obtained from Theorem \theoref{conv_pic}.}   
\figlab{conv_h}
\end{figure}

\subsubsection{Convergence with mesh size}
In Figure \figref{conv_h} we verify the convergence with respect to mesh size $h$ according to Theorem \theoref{conv_pic}. Since the cloud-in-cell shape function is second order ($m=1$), we expect a theoretical convergence of $\mc{O}\LRp{h^{2k}}$ from Theorem \theoref{conv_pic}.
We consider the Penning trap mini-app with ten particles per mode/cell as it has the least statistical noise out of the three mini-apps for the same number of total particles as explained before. From Figure \figref{conv_h_right}, we observe that the numerical results
follow the theoretical scaling during the initial iterations and/or larger $h$. However, for higher $k$ and/or smaller $h$, they deviate from the expected scaling. This is due to the error from the statistical noise term in Theorem \theoref{conv_pic}, which affects the scaling at lower levels of error. We were also able to verify the convergence with respect to $h$ for the Landau damping and two-stream instability test cases. But, due to the dominant statistical noise term in those
examples we had to use a much higher number of
particles per cell on $\mc{O}(1000)$ so that the grid-based error dominates. Similar to the Penning trap, we also observed in those cases, that the scalings deviate from the theoretical scalings at higher $k$ and/or smaller $h$ due to the noise. Similar to Figure \figref{conv_pc}, we observed that the velocity error curves follow the position error curves with a slightly ($\mc{O}(1)$ factor) higher magnitude. Hence, we do not show them in Figure \figref{conv_h} as well as in the following figures for brevity.     

\begin{figure}[h!b!t!]
    \subfigure[Error Vs Iterations]{
    \includegraphics[width=0.48\columnwidth]{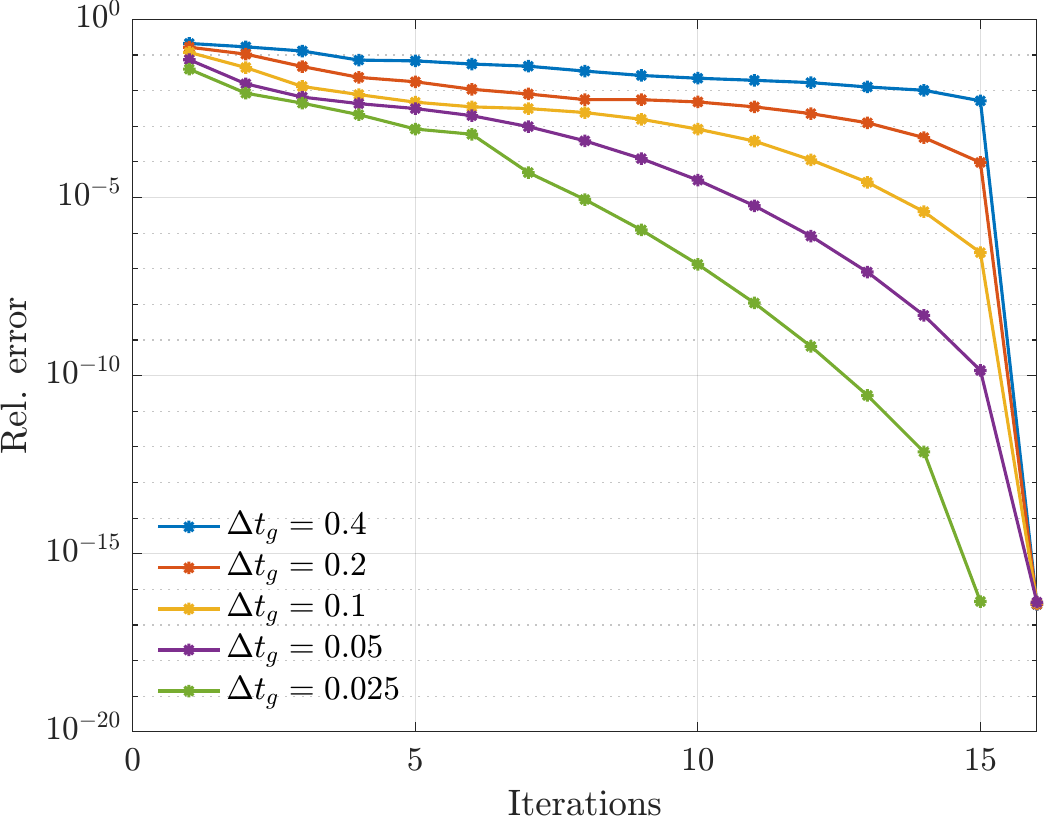}
  }
  \subfigure[Error Vs $\Delta t_g$]{
      \includegraphics[width=0.46\columnwidth]{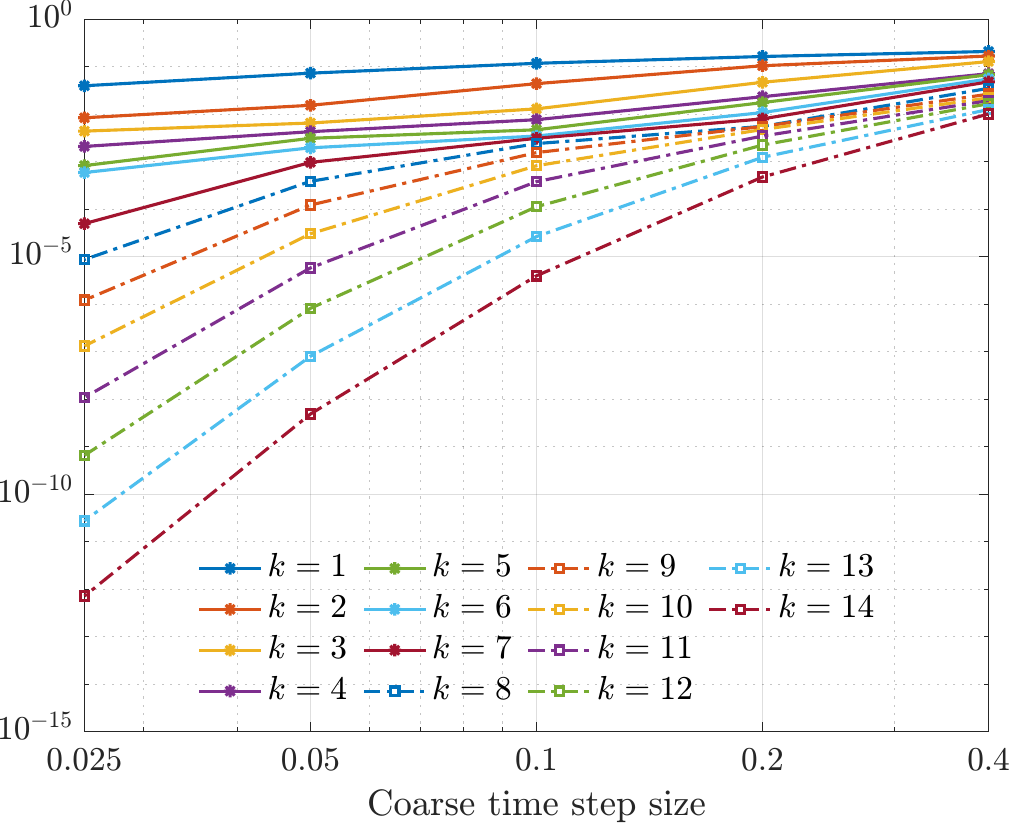}
      \figlab{conv_penning_long_right}
  }
  \subfigure[Error Vs Iterations]{
    \includegraphics[width=0.48\columnwidth]{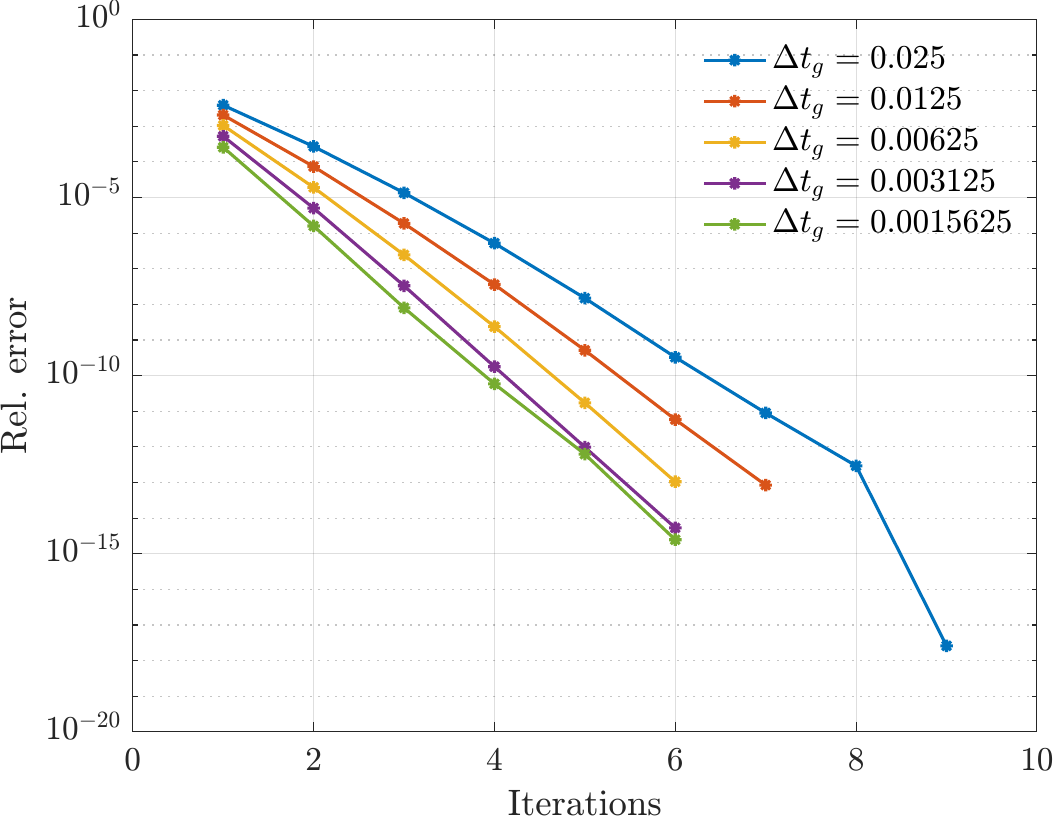}
  }
  \subfigure[Error Vs $\Delta t_g$]{
      \includegraphics[width=0.46\columnwidth]{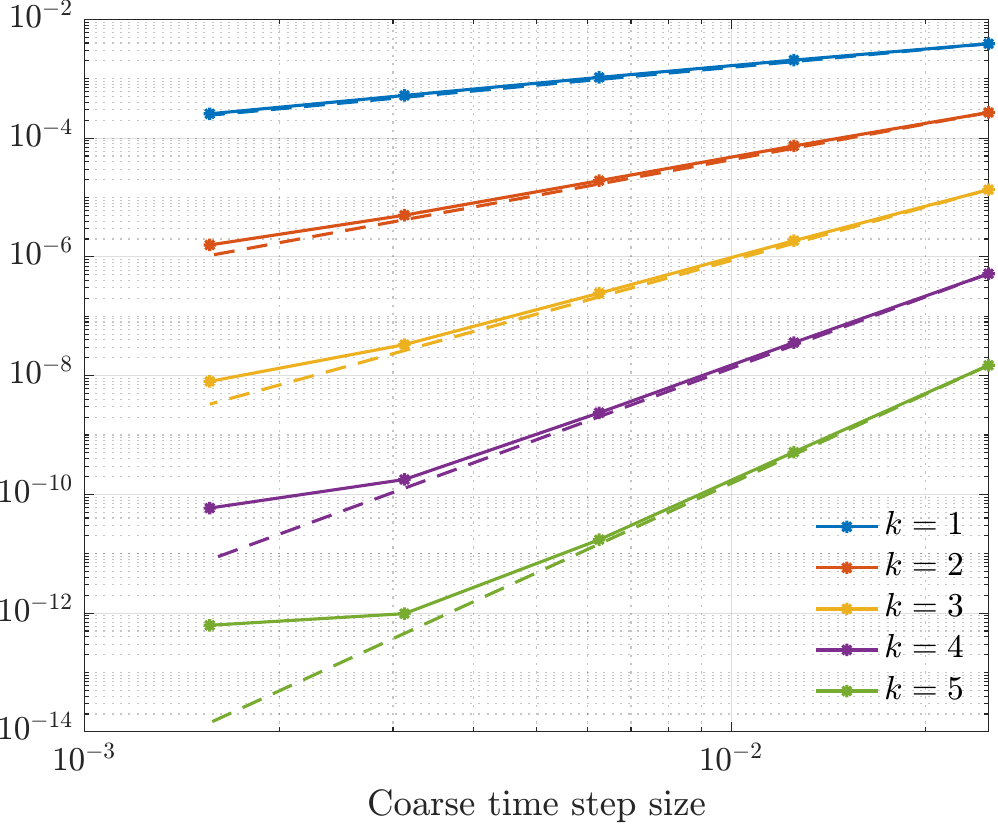}
      \figlab{conv_penning_small_right}
  }
    \caption{Penning trap: Relative error in position $\xb$ versus the coarse time step size $\Delta t_g$ and iterations $k$ with PIC as the coarse propagator. The parameters are $N_m=64^3$ and $P_c=10$. The final time $T=19.2$ and the fine time step size $\Delta t_f = 3.125\times10^{-3}$ for the top row whereas it is 16 times smaller, i.e., $T=1.2$ and $\Delta t_f = 1.953125\times10^{-4}$ for the bottom row. The dashed lines in the right column of the bottom row represent the scaling $\mc{O}\LRp{\Delta t_g^{k}}$ which is less than $\mc{O}\LRp{\Delta t_g^{2k}}$ obtained from Theorem \theoref{conv_pic}.}   
\figlab{conv_dt_pic_penning}
\end{figure}

\subsubsection{Convergence with coarse time step size}
\seclab{conv_time_step_size}
We now investigate the convergence with respect to the coarse time step size $\Delta t_g$ using PIC and PIF as coarse propagators. The velocity Verlet or the Boris time integrator that we use is second order, i.e., $p=2$, in Theorems \theoref{conv_pic} and \theoref{conv_pif}. Now, the product term $\prod_{j=1}^{k}(n+1-j)$ can be bounded by $n^k$ and multiplying that term by the term $\LRp{\Delta T \Delta t_g^p}^k$ in Theorems \theoref{conv_pic}, \theoref{conv_pif} gives $\LRp{T \Delta t_g^p}^k$. Thus, we expect a theoretical convergence of $\mc{O}\LRp{\Delta t_g^{2k}}$ for our case.  

We first consider the Penning trap mini-app with $N_m=64^3$, $P_c=10$, $T=19.2$ and $\Delta t_f = 3.125\times10^{-3}$. From Figure \figref{conv_penning_long_right} we observe that it does not follow the theoretical scaling of $\mc{O}\LRp{\Delta t_g^{2k}}$ and shows convergence only 
at higher $k$ and smaller $\Delta t_g$. The Penning trap has three frequencies leading to three different time scales. The fastest one is the 
modified cyclotron frequency $\omega_+$, followed by the axial frequency $\omega_z$, and the slowest one being the magnetron frequency $\omega_-$ \cite{blaum2022}. For our selected parameters these values are $\omega_+ = 4.875$, $\omega_z=1.1$ and $\omega_- = 0.125$ which corresponds to time periods of $1.29$, $5.7$ and $50.3$ respectively. The higher modified cyclotron frequency as visible in Figure \figref{penning_pe} leads to a large
constant $C_{time}$ so that we need high $k$ and small $\Delta t_g$ in Theorem \theoref{conv_pic} to obtain convergence. It should be noted that since the number of time subdomains is $16$ this leads to virtually no speedup with the parareal algorithm compared to the time serial case. 

In order to reduce the effect of constant $C_{time}$ in the convergence behavior, in the bottom row of Figure \figref{conv_dt_pic_penning} we use an end time $T=1.2$, which is $16$ times smaller than the previous case of $T=19.2$. We correspondingly reduced the fine time step size and the coarse time step sizes by the 
same factor in order to perform the convergence study. In Figure \figref{conv_penning_small_right} we observe a better convergence compared to Figure \figref{conv_penning_long_right}. However, the scaling is only $\mc{O}\LRp{\Delta t_g^k}$ instead of the theoretically predicted $\mc{O}\LRp{\Delta t_g^{2k}}$. At the time of writing, we do not yet know the reason for this order reduction and we noticed that further reducing the final time and time steps sizes did not help in improving the order.
We also observed similar convergence behavior when using PIF as coarse propagator for the Penning trap, and hence the behavior is linked to the dynamics of the test case rather than the choice of the coarse propagator. 

For the  Landau damping test case with PIF as the coarse propagator we observe the theoretical scaling of $\mc{O}\LRp{\Delta t_g^{2k}}$ in 
Figure \figref{conv_dt_pif_landau}, whereas, with PIC as the coarse propagator, the error is completely dominated by the statistical noise. Hence, we do 
not observe any convergence with respect to the coarse time step size in Figure \figref{conv_dt_pic_landau}. Thus in this case it is not the dynamics of the test 
case but rather the choice of the coarse propagator and the associated number of particles, mesh size which lead to different convergence behaviors. 
For the two-stream instability test case, we observed similar convergence behavior with respect to coarse time step sizes and the choice of the coarse propagator as that of the Landau damping. 

\begin{figure}[h!b!t!]
    \subfigure[Error Vs Iterations]{
    \includegraphics[width=0.48\columnwidth]{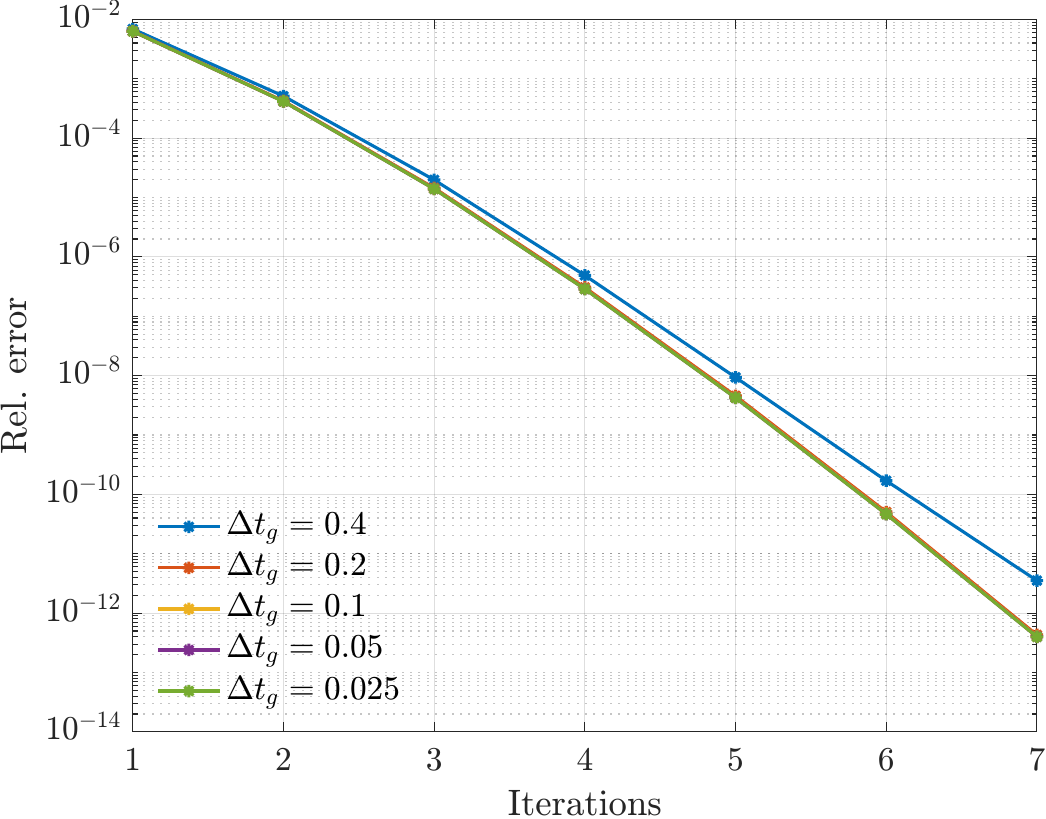}
  }
  \subfigure[Error Vs $\Delta t_g$]{
    \includegraphics[width=0.46\columnwidth]{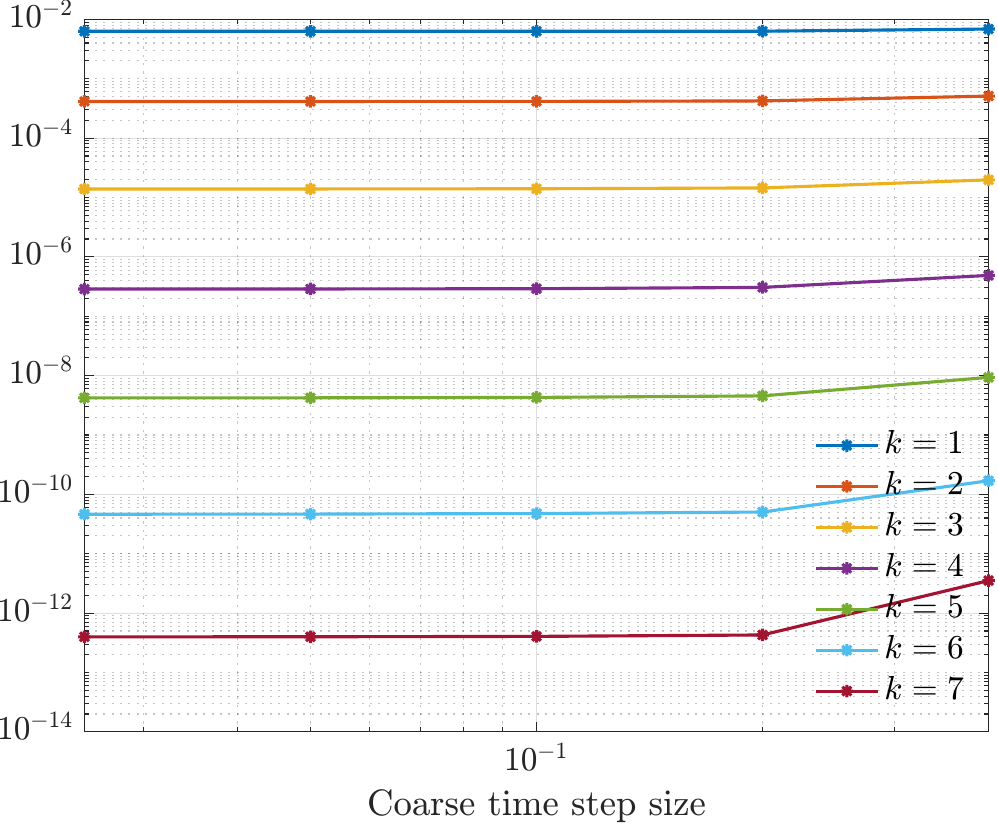}
    \figlab{conv_dt_pic_landau}
  }
  \subfigure[Error Vs Iterations]{
    \includegraphics[width=0.48\columnwidth]{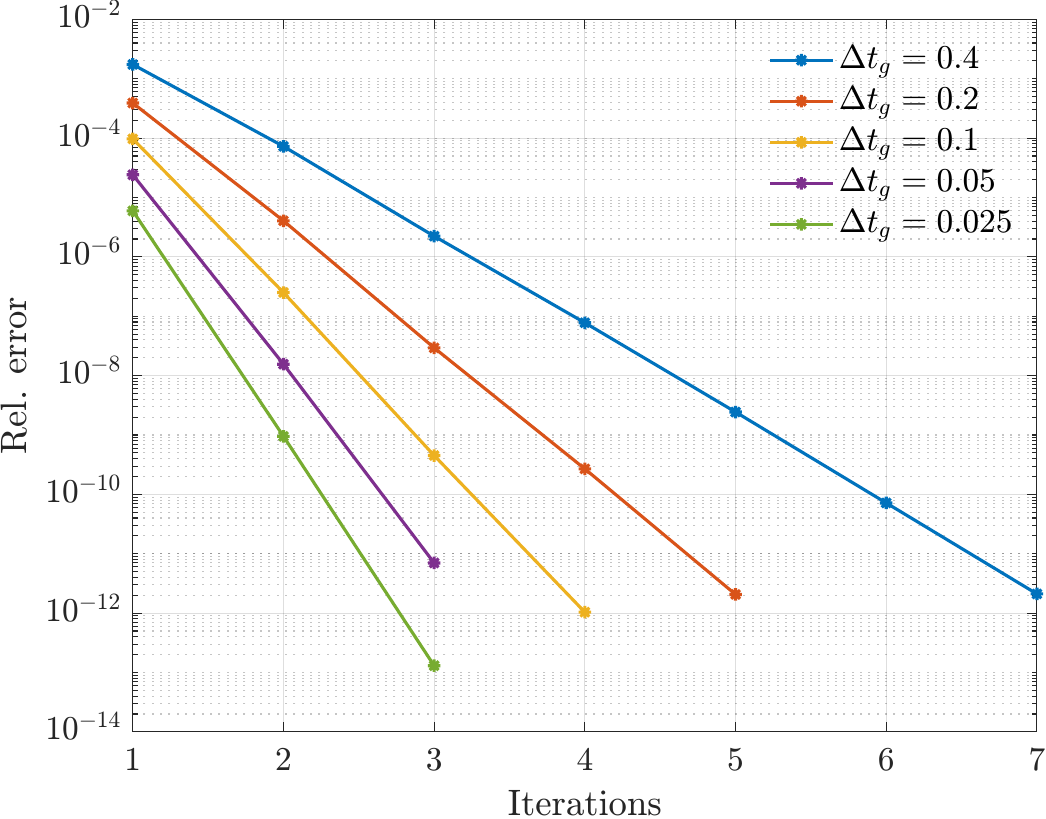}
  }
  \subfigure[Error Vs $\Delta t_g$]{
    \includegraphics[width=0.46\columnwidth]{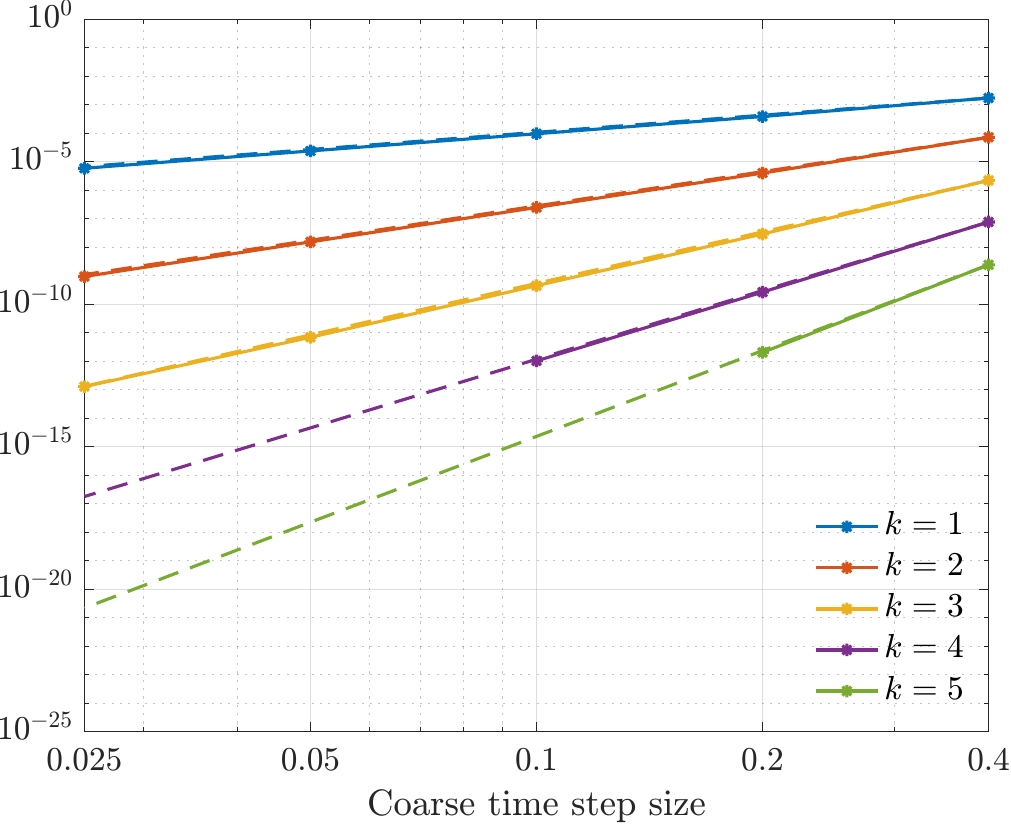}
    \figlab{conv_dt_pif_landau}
  }
    \caption{Landau damping: Relative error in position $\xb$ versus the coarse time step size $\Delta t_g$ and iterations $k$ with PIC as the coarse propagator (top row) and NUFFT PIF as the coarse propagator (bottom row). The parameters are $N_m=16^3$, $P_c=640$, $T=19.2$ and $\Delta t_f = 3.125\times10^{-3}$. The dashed lines in the right column of bottom row represent the theoretical scaling $\mc{O}\LRp{\Delta t_g^{2k}}$ obtained from Theorem \theoref{conv_pif}. The NUFFT tolerance for the fine propagator in the top row is $10^{-12}$ whereas both the fine and the coarse NUFFT tolerances in the bottom row are chosen as $10^{-6}$ to eliminate the spatial error.}   
\figlab{conv_dt_pic_pif_landau}
\end{figure}

\begin{figure}[h!b!t!]
    \subfigure[Error Vs Iterations]{
    \includegraphics[width=0.48\columnwidth]{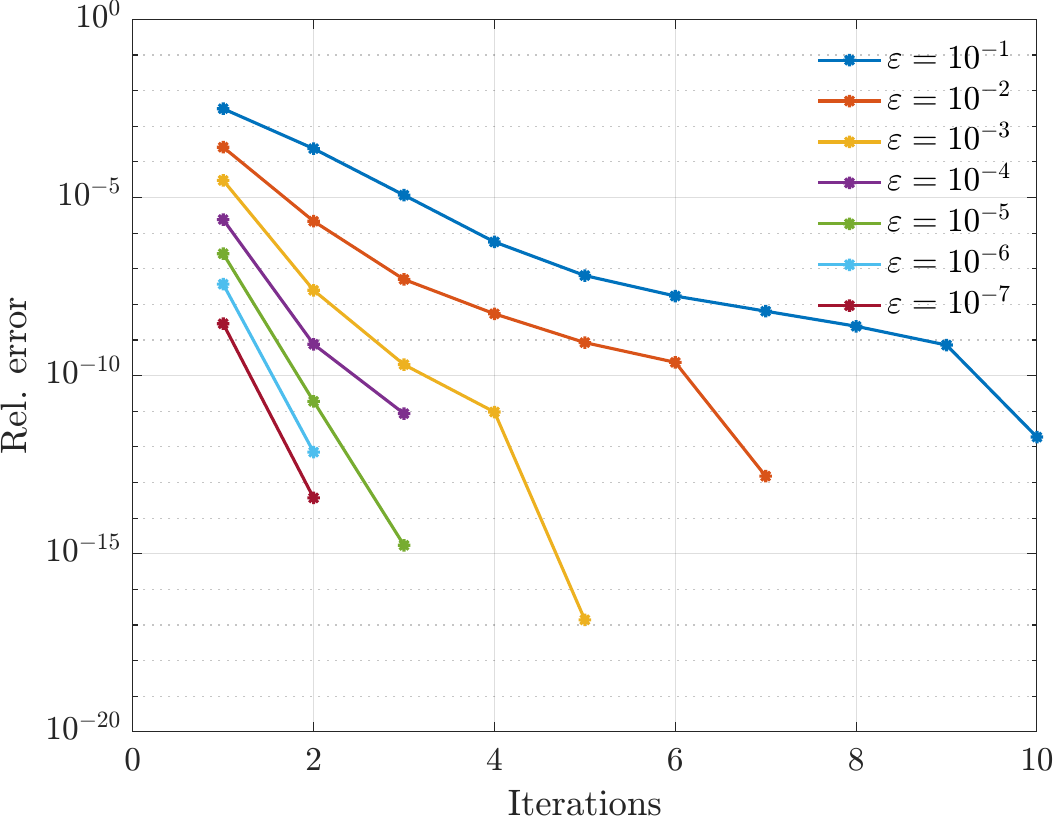}
  }
  \subfigure[Error Vs $\varepsilon$]{
    \includegraphics[width=0.46\columnwidth]{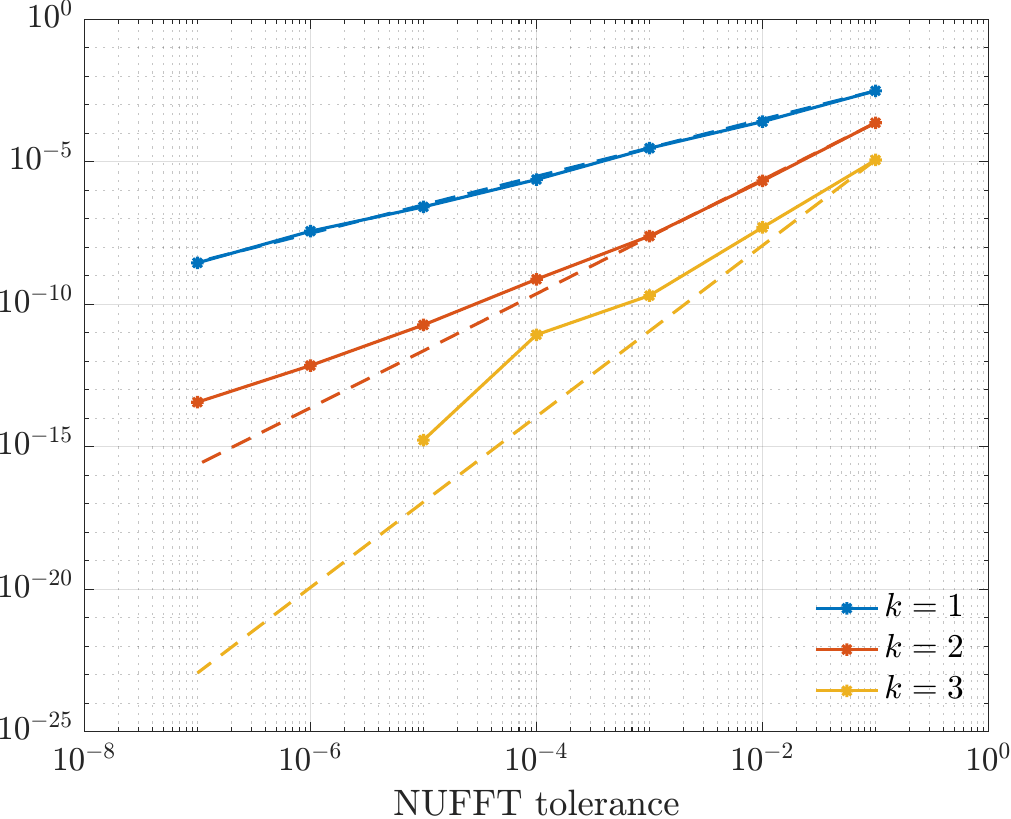}
    \figlab{conv_epsilon_right}
  }
    \caption{Landau damping: Relative error in position $\xb$ versus the NUFFT tolerance $\varepsilon$ and iterations $k$ with NUFFT PIF as the coarse propagator. The parameters are $N_m=16^3$, $P_c=640$, $T=19.2$ and $\Delta t_f = \Delta t_g = 0.05$. The dashed lines in the right figure represent the theoretical scaling $\mc{O}\LRp{\varepsilon^{k}}$ obtained from Theorem \theoref{conv_pif}.}   
\figlab{conv_epsilon}
\end{figure}

\subsubsection{Convergence with coarse NUFFT tolerance}
Finally, in Figure \figref{conv_epsilon} we verify the convergence with respect to NUFFT tolerance when PIF is used as the coarse propagator for the Landau damping mini-app. From Figure \figref{conv_epsilon_right} we see good agreement of the numerical results with the predicted theoretical scaling of $\mc{O}\LRp{\varepsilon^k}$ from Theorem \theoref{conv_pif}. We were also able to verify similar convergence results for the two-stream instability as well as the Penning trap mini-apps with PIF as the coarse propagators.

\subsection{Parallel scaling study}

We select the best coarse propagator parameter combinations for each of the mini-apps by performing a comprehensive parameter study in Appendix \secref{parameter_study}. We choose the tolerances for fine propagator NUFFT as well as parareal by considering the errors in the conservation of different quantities in Appendix \secref{conservation}. Below is a summary of inferences from the parameter study:

\begin{itemize}
    \item For larger fine time step sizes and coarser NUFFT tolerances in the fine propagator, PIC (with possible time coarsening) gives the 
    least time to solution.
    \item If we want highly accurate solutions with small fine time step sizes and stricter NUFFT tolerances, PIF with a coarse NUFFT tolerance is the best combination if time coarsening is possible.
    \item If the test case does not allow much time coarsening (e.g. in the Penning trap example) then PIC on the coarse level gives the least time to solution.
\end{itemize}
 Even though the inferences are made only with respect to the mini-apps considered, they mostly generalize to other test cases, since the mini-apps are selected to be representative of different scenarios in kinetic plasma simulations. 

\begin{figure}[h!b!t!]
\centering
    \includegraphics[width=0.75\columnwidth]{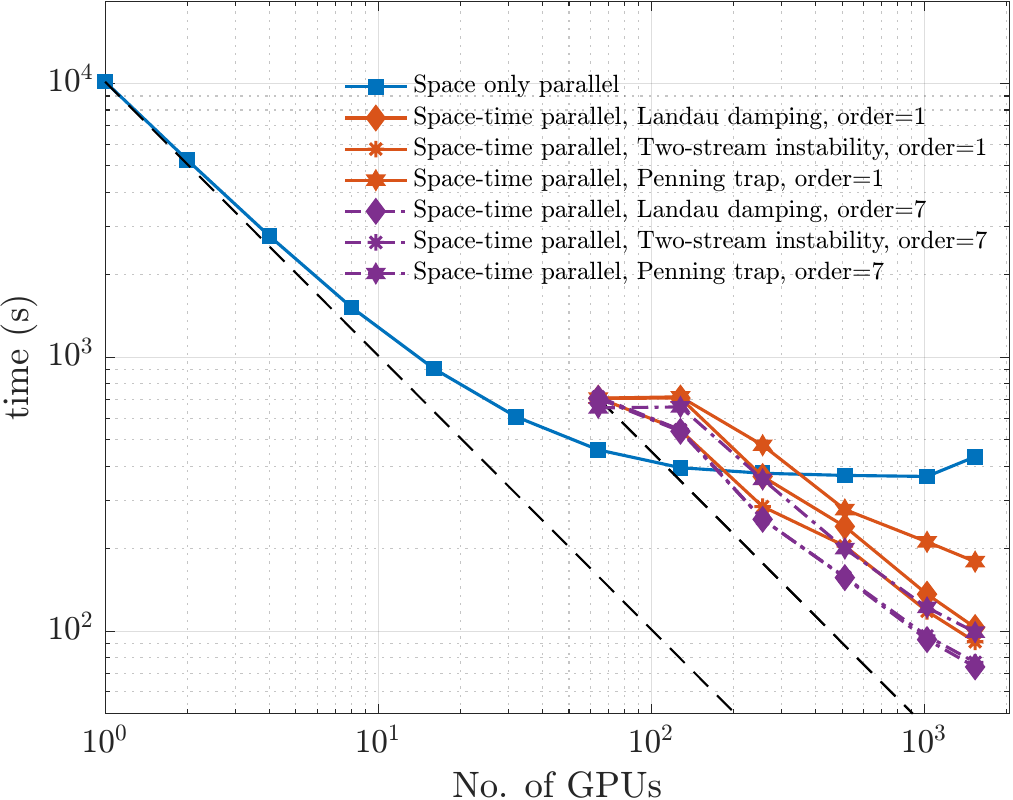}
    \caption{Timings for space-time parallelization versus spatial parallelization alone for all the mini-apps. The fine propagator is the PIF scheme with $128^3$ modes, $10$ particles per mode and time step size $\Delta t_f = 0.003125$.}   
    \figlab{scaling_study}
\end{figure}

With the guidance of the parameter study in Figures \figref{landau_heatmap} and \figref{penning_heatmap} we perform a parallel scaling study
for all the mini-apps. Linear or cloud-in-cell shape function is used and the problem size is $N_m=128^3$, $P_c=10$ which amounts to $N_p=20,971,520$. We take the fine time step size to be 
$\Delta t_f=0.003125$ and $T=19.2$. As explained in Appendix \secref{heatmaps} we switch to space-time parallelization and start
allocating resources in the time parallelization after the efficiency drops below $50\%$ for the spatial parallelization. For this problem size it happens after $32$ GPUs.
Hence, we take $32$ GPUs for spatial parallelization and use $2,4,8,16,32$ and $48$ GPUs along time direction for the scaling study.
The spatial parallelization of PIF schemes is more or less independent of the test case. In order to reduce the usage of core hours we ran only 
the Penning trap mini-app and use it as a reference to compare against the space-time parallelization of all three mini-apps. In a similar effort to 
reduce the amount of computing time, since the timing of the spatial parallelization of PIF schemes scales linearly with the number of time steps, and relatively independent of the test case, we ran the Penning trap simulation for 768 time steps ($T=2.4$) and
scale it by a factor of $8$ to get the final reference timing for $6144$ time steps ($T=19.2$).

For the coarse propagator we take PIF with $\varepsilon=10^{-3}$ and $\Delta t_g=0.05$ for the Landau damping and PIF 
with $\varepsilon = 10^{-4}$ and $\Delta t_g=0.05$ for the two-stream instability test cases based on the parameter study in Appendix \secref{heatmaps}. For both these cases we use the parareal posed on the entire time domain. For the Penning trap, we take PIC as coarse propagator with $16$ blocks for the reasons explained in Section \secref{multiblock}. From numerical experiments we found that we can coarsen the time step size a bit for the Penning trap for this problem size and hence we select $\Delta
t_g=0.0125$ (coarsened by a factor 4 with respect to the fine time step size). The strong scaling curves for the combined space-time parallelization are shown in red in Figure
\figref{scaling_study}. We can see that for $2$ and $4$ GPUs in the time parallelization, the space-time parallelization does not pay off and it takes
more time than spatial parallelization alone. This is because, for the parareal algorithm to have speedup for $2$ and $4$ subdomains it need to 
converge in $1$ and less than $4$ iterations which does not typically happen. The cross-over point for space-time parallelization happens at 
$8$ GPUs where it takes comparable time to spatial parallelization alone. Starting from $16$ GPUs onward we start to see speedup for all the mini-apps
compared to spatial parallelization alone. The maximum speedup obtained at $1536$ ($32\times48$) GPUs is $4.2$, $4.7$ and $2.4$ for Landau damping, two-stream instability and Penning trap mini-apps respectively. From the timings of the individual components in the scaling study we identified that for large number of modes and at high number of GPUs for the spatial parallelization the \texttt{MPI\_Allreduce} step after scatter becomes a dominant factor. This cost is the same for the fine and coarse propagators without time coarsening. Hence, it is one of the major reasons for the smaller speedup in the case of the Penning trap as we are able to coarsen the time step size only by a factor of $4$ compared to $16$ in the other two mini-apps. 

\subsubsection{PIF with higher-order shape function}
One of the advantages of PIF schemes compared to PIC schemes as mentioned in \cite{mitchell2019efficient} is that higher-order 
shape functions can be used at a similar computational cost as lower-order shape functions. This is because they are precomputed in the Fourier space at 
the beginning and reused for all the time steps. Hence, we evaluate here the scaling with a higher-order shape function in the fine PIF propagator.
Now, if PIC is used as a coarse propagator we would still like to use the linear shape function
and select the mesh size $h$ such that we approximate the Fourier spectrum of the higher-order PIF shape function. This has two advantages: i) the cost of 
the coarse propagator does not increase ii) it typically leads to larger $h$ or less number of grid points in the PIC scheme which in turn reduces
the \texttt{MPI\_Allreduce} cost of spatial parallelization in the coarse propagator compared to the fine propagator. For large number of modes/grid cells and
high GPU counts for spatial parallelization the latter advantage is significant as we see from the Penning trap test case as follows. We select 
an order $7$ B-spline function for the PIF scheme in the fine propagator and it has a Fourier spectrum close to linear shape function of mesh size $2h$. Hence for the Penning trap test case we select PIC with mesh size $2h$, $16$ blocks and no coarsening in time step size $\Delta t_g=0.003125$ (based on numerical experiments) as the 
coarse propagator and PIF with order $7$ B-spline as the fine propagator. For the Landau damping and two-stream instability test cases we take the same 
PIF coarse propagator combinations as before but with order $7$ B-spline shape functions. We now observe that the purple scaling curves for the space-time 
parallelization are better than the red ones in Figure \figref{scaling_study} for all the mini-apps. The maximum speedup obtained at $1536$ ($32\times48$) GPUs for the higher-order shape function is $5.8$, $5.6$ and $4.4$ for Landau damping, two-stream instability and Penning trap mini-apps respectively. The 
main reason for the improved speedups in the Landau damping and two-stream instability mini-apps is the reduced noise from higher-order shape function which
in turn reduces the constant $C_{time}$ in Theorem \theoref{conv_pif}. For the Penning trap it is mainly due to the reduced cost associated with the \texttt{MPI\_Allreduce}
operation in the coarse propagator as the number of grid points is now $8$ times smaller compared to the linear shape function with mesh size $h$ before.
One more interesting avenue to reduce the cost associated with the \texttt{MPI\_Allreduce} operation in the coarse propagator is to use single precision for the density 
field which reduces the communication volume. We intend to pursue that direction in our future work. 

The scaling studies show great promise for speedups with space-time parallelization compared to spatial parallelization alone.
\emph{We get a particle push rate of approximately one billion particles per second in all the test cases 
from Figure \figref{scaling_study} at 1536 GPUs.} This makes PIF schemes a viable candidate for large-scale, 
production-level applications in plasma physics and beyond, suited to exascale computing architectures.  

\section{Conclusions}
\seclab{conclusions}

In this work we present a parareal approach in the phase-space for time parallelization of particle-in-Fourier schemes applied to Vlasov-Poisson system of 
equations in kinetic plasma simulations. We choose the coarse propagators for the parareal algorithm by selecting 
PIF with NUFFT of lenient tolerance or the standard PIC scheme. They are accompanied with or without coarsening of time step sizes depending on the test case and application. We perform an error analysis for the parareal algorithm with explicit dependence on the coarse 
discretization parameters such as the mesh size, particles per cell, NUFFT tolerance and the coarse time step size.

We verify the theoretical results numerically using Landau damping, two-stream instability, and the Penning trap, which are some of 
the standard benchmark problems in kinetic plasma physics simulations. We observe that for test cases with oscillatory solutions and multiple time scales 
such as the Penning trap coarsening in time leads to large errors and high number of iterations. Also, parareal used on the entire time domain in these
cases converges very slowly and hence needs to be applied in multiple blocks. 

Finally, we conduct a scaling
study up to $1536$ A100 GPUs and observe that overall the space-time parallelization gives approximately 
$4-6$ times speedup for Landau damping and two-stream instability test cases and $2-4$ times speedup for the Penning trap test case compared to spatial parallelization alone. We achieve a push rate of around 1 billion particles per second for all the test cases at 1536 GPUs. Thus the space-time parallel PIF schemes show promise for large-scale kinetic plasma simulations with excellent stability and conservation properties.

In terms of future works, we are planning to perform massively space-time parallel PIF simulations with a focus on HPC optimization and scalability. A natural extension of the parareal algorithm for electrostatic PIF schemes is to extend it to electromagnetic PIF schemes for simulating
Vlasov-Maxwell systems. However, that involves Maxwell's equations which are time dependent and hyperbolic in nature and this may present an issue 
for the parareal algorithm. In those cases the parareal scheme for 
the particle equations can be combined with parallel-in-time schemes 
suited for purely hyperbolic equations such as ParaExp and 
ParaDiag \cite{gander2013paraexp,gander2020paradiag}. Recently, a Fourier spectral 
scheme similar to PIF has been proposed in the context of immersed boundary methods
in \cite{chen2024fourier}. Our algorithm can be applied for this scheme, too, and hence 
can find application in the field of fluid structure interactions in 
computational biology.

\appendix
\section{Parameter study}
\seclab{parameter_study}
We consider PIF and PIC coarse propagators along with different coarse time step sizes and find the parameter combinations which provide the maximum speedup for the parareal algorithm. For the parameter study we consider the following problem size: $N_m = 64^3$, $P_c=10$, end time $T=19.2$ and two fine time step sizes $\Delta t_f = 0.05$ and $\Delta t_f=0.003125$. The reason for the selection of a coarse and fine time step size is due to the associated energy conservation which improves with the decrease
in time step size as $\mc{O}\LRp{\Delta t_f^2}$ for the velocity Verlet or the Boris time integrator. Hence, these two scenarios represent applications which may require different levels of energy conservation. 

\subsection{Verification}
\begin{figure}[h!b!t!]
    \subfigure[Landau damping]{
    \includegraphics[width=0.305\columnwidth]{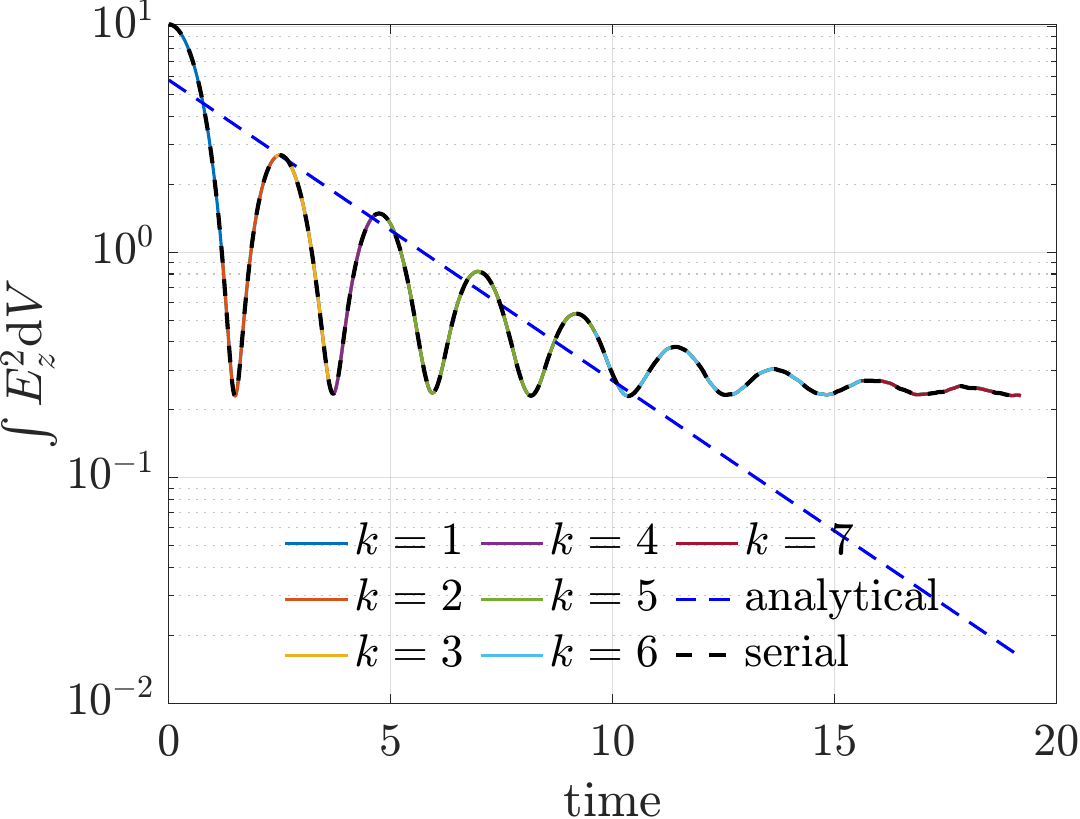}
  }
    \subfigure[Two-stream instability]{
    \includegraphics[width=0.305\columnwidth]{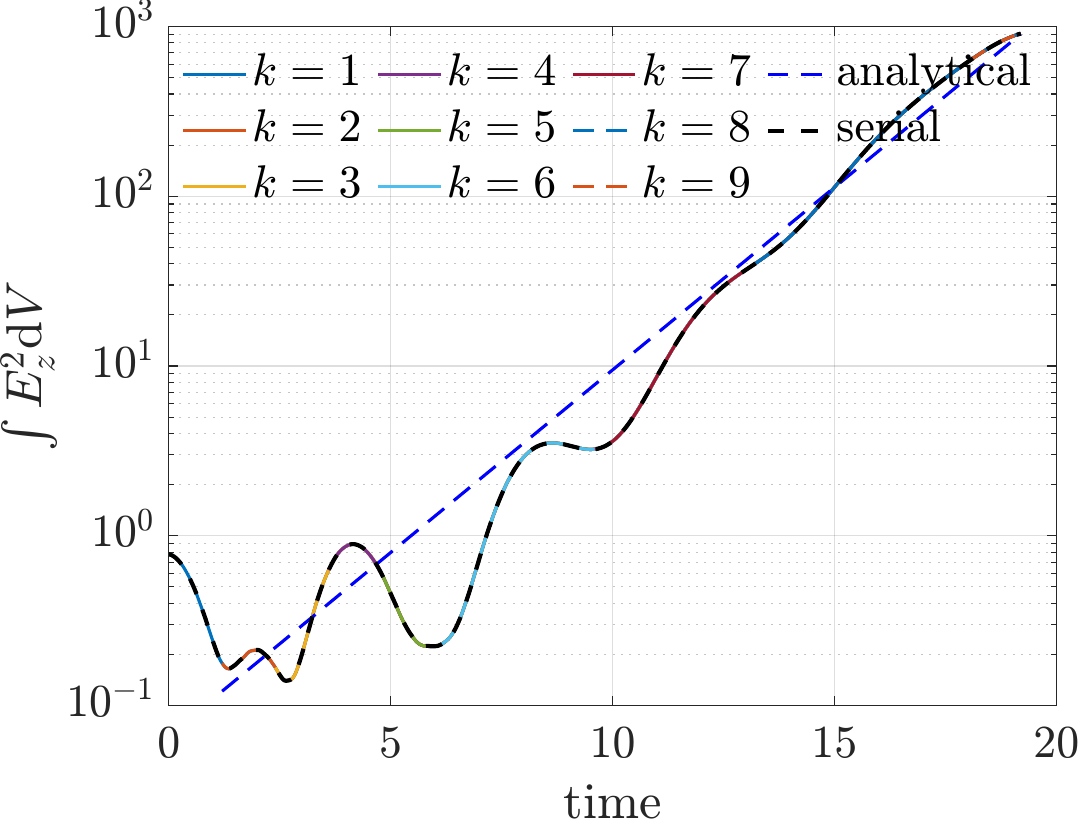}
  }
    \subfigure[Penning trap]{
    \includegraphics[width=0.305\columnwidth]{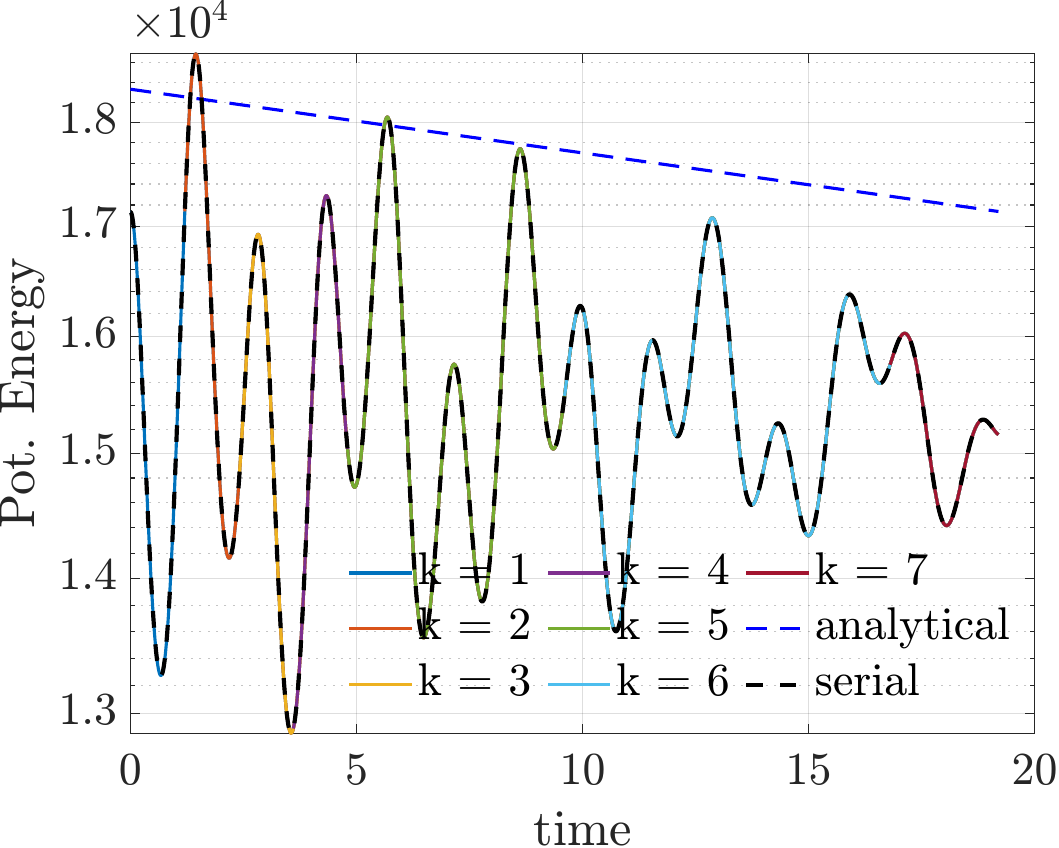}
    \figlab{penning_pe}
  }
    \caption{Electric field energy in the $z-$direction versus time for the Landau damping (left) and the two-stream instability (center) test cases 
    and potential energy versus time for the Penning trap (right) test case. The number of modes and the number of particles per mode are 
    $64^3$ and $10$ respectively. PIF with NUFFT tolerance $10^{-7}$ and fine time step size $\Delta t_f=0.003125$ is used as the fine propagator. 
    PIC with coarse time step size $\Delta t_g=0.05$ is used for the Landau damping and the 
    two-stream instability test cases, whereas for Penning trap, PIC with $\Delta t_g = 0.003125$ is used as the coarse propagator. The number of time subdomains is $16$ and the stopping tolerance for parareal is $10^{-8}$.}
    \figlab{sol_all_mini}
\end{figure}

In Figure \figref{sol_all_mini}, we show the electric field energy in the \\ $z-$direction versus time for the Landau damping and the two-stream instability mini-apps and the potential energy versus time for the Penning trap mini-app for the case of $\Delta t_f = 0.003125$. PIC with a coarse time
step size of $\Delta t_g=0.05$ for the Landau damping and the two-stream instability and $\Delta t_g=\Delta t_f$ for the Penning trap is used as the
coarse propagator. As we can see from Figure \figref{sol_all_mini} starting from the first iteration of parareal itself the quantities 
match well with the serial time stepping as well as the analytical rates from the dispersion relation. 

\subsection{Conservation}
\seclab{conservation}
\begin{figure}[h!b!t!]
    \subfigure[Landau damping]{
    \includegraphics[width=0.47\columnwidth]{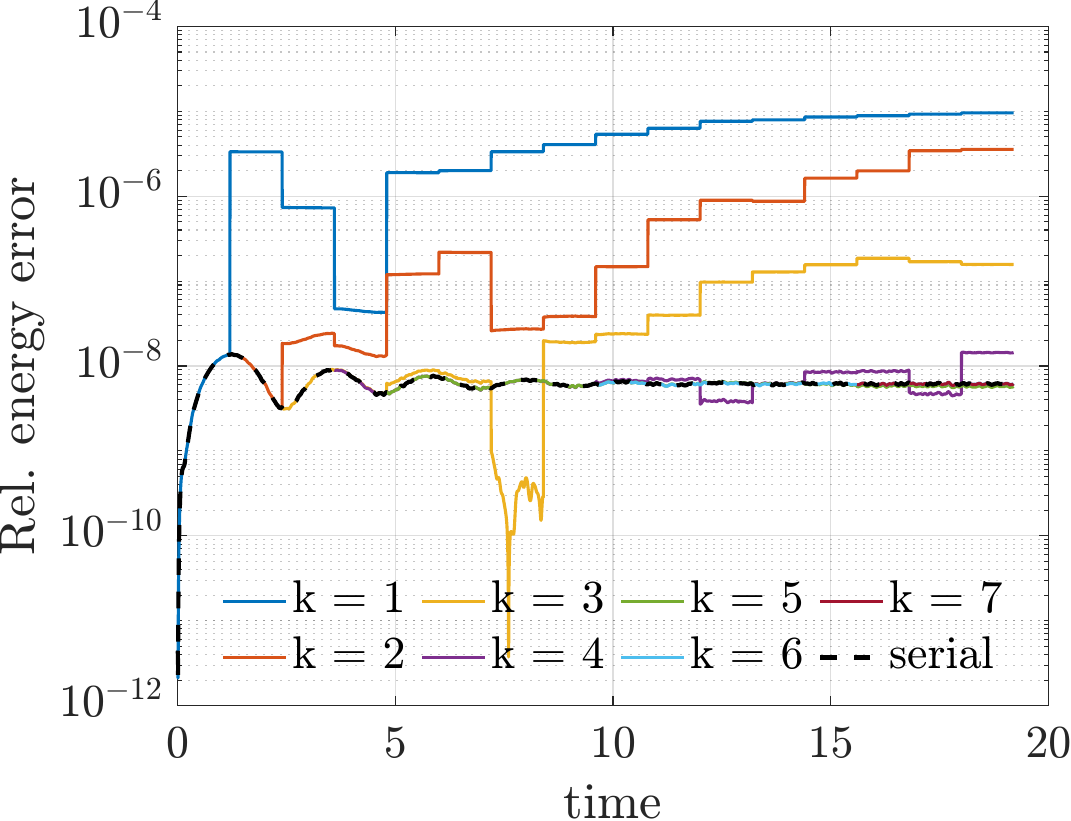}
    \figlab{landau_energy}
  }
    \subfigure[Two-stream instability]{
    \includegraphics[width=0.47\columnwidth]{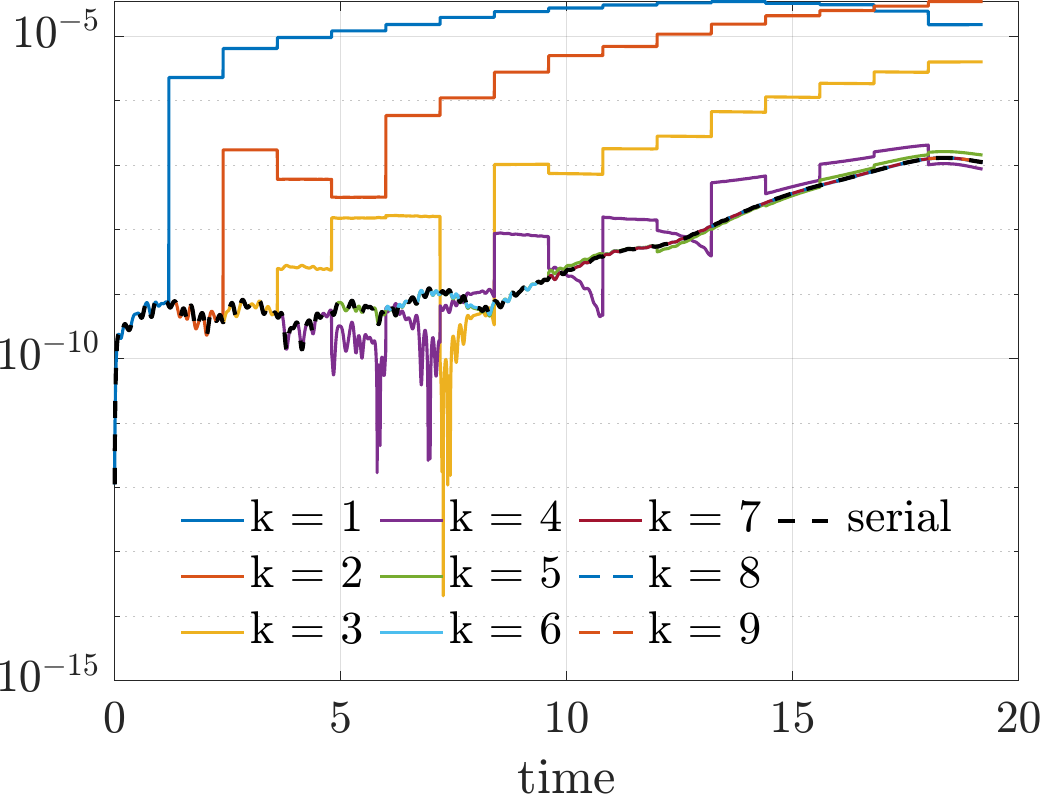}
    \figlab{twostream_energy}
  }
    \caption{Relative energy error for the Landau damping (left) and the two-stream instability (right) test cases.} 
    \figlab{landau_twostream_energy}
\end{figure}

\begin{figure}[h!b!t!]
    \subfigure[Penning trap]{
    \includegraphics[width=0.3\columnwidth]{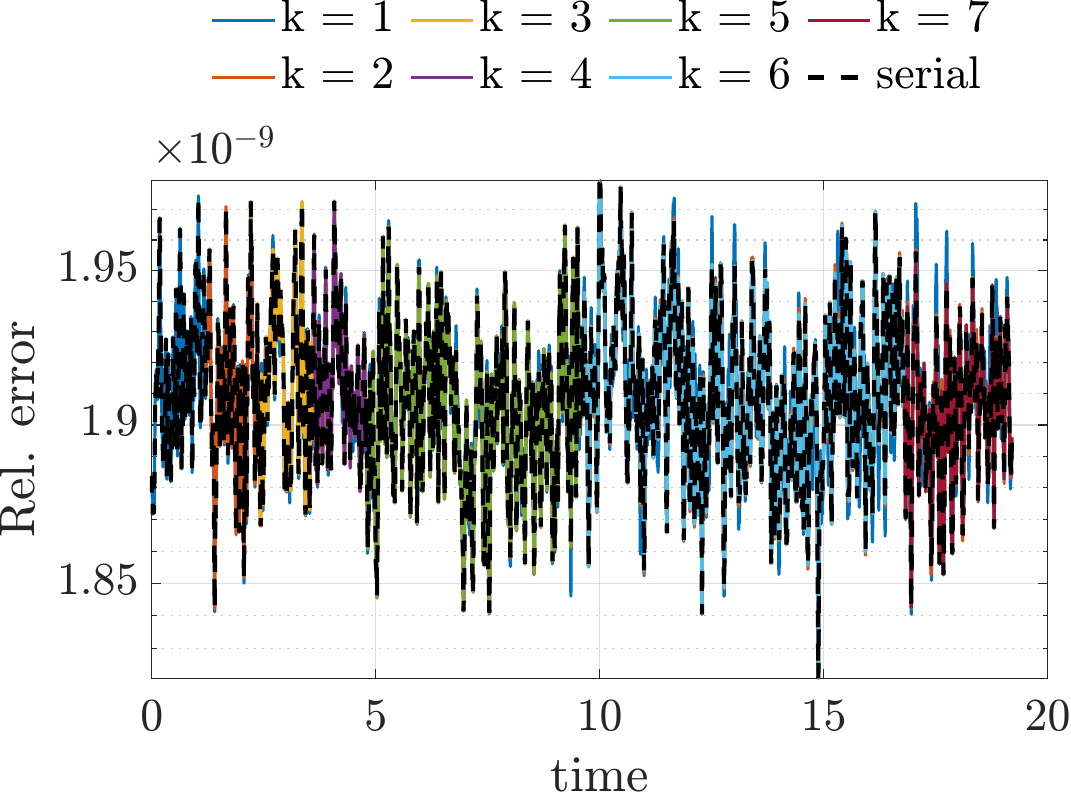}
    \figlab{penning_charge}
  }
    \subfigure[Landau damping]{
    \includegraphics[width=0.3\columnwidth]{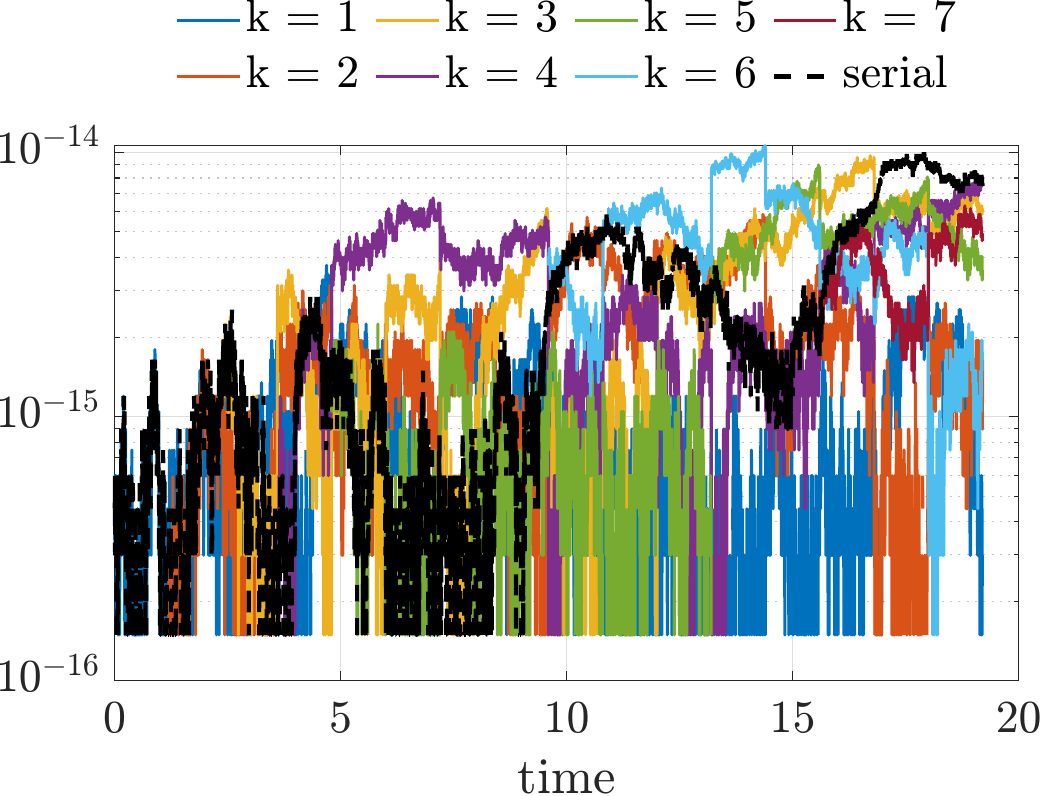}
    \figlab{landau_momentum}
  }
    \subfigure[Two-stream instability]{
    \includegraphics[width=0.3\columnwidth]{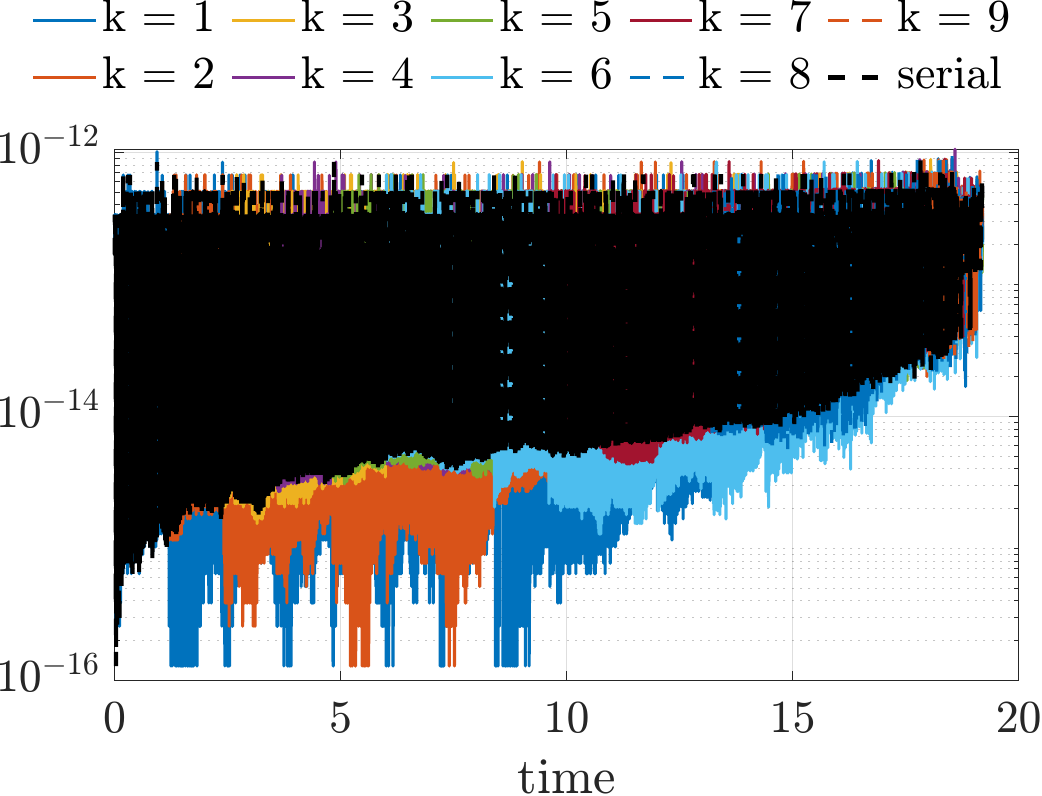}
    \figlab{twostream_momentum}
  }
    \caption{Relative charge error in the Penning trap (left) and momentum errors in the Landau damping (center) and the two-stream instability (right) test cases.} 
    \figlab{charge_momentum_errors}
\end{figure}

In Figures \figref{landau_twostream_energy} and \figref{charge_momentum_errors} we show the conservation of different quantities for the three test cases.
From Figures \figref{landau_energy} and \figref{twostream_energy} the relative error in energy conservation differs for different parareal iterations and becomes similar to the serial time stepping only close to convergence. Since PIC is used as the coarse propagator and it is not energy conserving, we can see that the
relative error in the earlier iterations is orders of magnitude higher and has an increasing behavior than the error obtained in the event of convergence. \emph{This figure clearly shows the benefit of PIF schemes over the standard PIC scheme with regards to improved and stable relative errors in energy conservation which is very important for long time integration simulations.} We note here that there are many other flavors of PIC which differ in terms of field solvers, shape
functions, sampling approaches and noise reduction strategies. Investigating them as coarse propagators for PIF is an interesting avenue which will be carried out elsewhere in the future. Now, in terms of momentum conservation we always have a relative error close to machine precision, which
is independent of the parareal iterations and the NUFFT tolerance as shown in Figures \figref{landau_momentum} and \figref{twostream_momentum}. 
The charge conservation, on the other hand, depends on the tolerance of the NUFFT for the fine propagator, whereas, the
energy conservation depends both on the NUFFT tolerance as well as the stopping tolerance for the parareal algorithm.
We show the charge conservation error for the Penning trap test case in Figure \figref{penning_charge}, which is very similar for the Landau damping and two-stream instability test cases and therefore not shown. For the case of the Penning trap, we add external electric and magnetic fields, and the external electric field also depends on the particle positions as
shown in equation \eqnref{penning_ext_efield}. This makes it difficult to measure the relative errors in energy and momentum conservation. However, we verified that the total energy and momentum magnitudes match well with the serial results.
We select the tolerances for the parareal and the fine propagator NUFFT such that the orders of magnitude of errors in the energy and charge conservation are comparable and match that of the serial time stepping while the momentum conservation is always satisfied irrespective of them. This is important because a tight tolerance for the fine propagator NUFFT and a lenient tolerance for 
parareal can lead to a large speedup of the parareal algorithm with respect to serial time stepping, whereas, a tight tolerance for parareal and a lenient
tolerance for NUFFT can lead to less or no speedup. Thus selecting certain combinations can lead to bloated speedups for the parareal algorithm as mentioned in \cite{gotschel2020twelve} 
without bringing in any additional value for the application. Here, we select a tolerance of $10^{-7}$ for the fine propagator NUFFT and $10^{-8}$ for the parareal, for the case of $\Delta t_f=0.003125$, based on the relative errors in the conservation in Figures \figref{landau_twostream_energy} and \figref{charge_momentum_errors}.       

\subsection{Convergence of error across time subdomains}
\begin{figure}[h!b!t!]
    \subfigure[Landau damping]{
    \includegraphics[width=0.31\columnwidth]{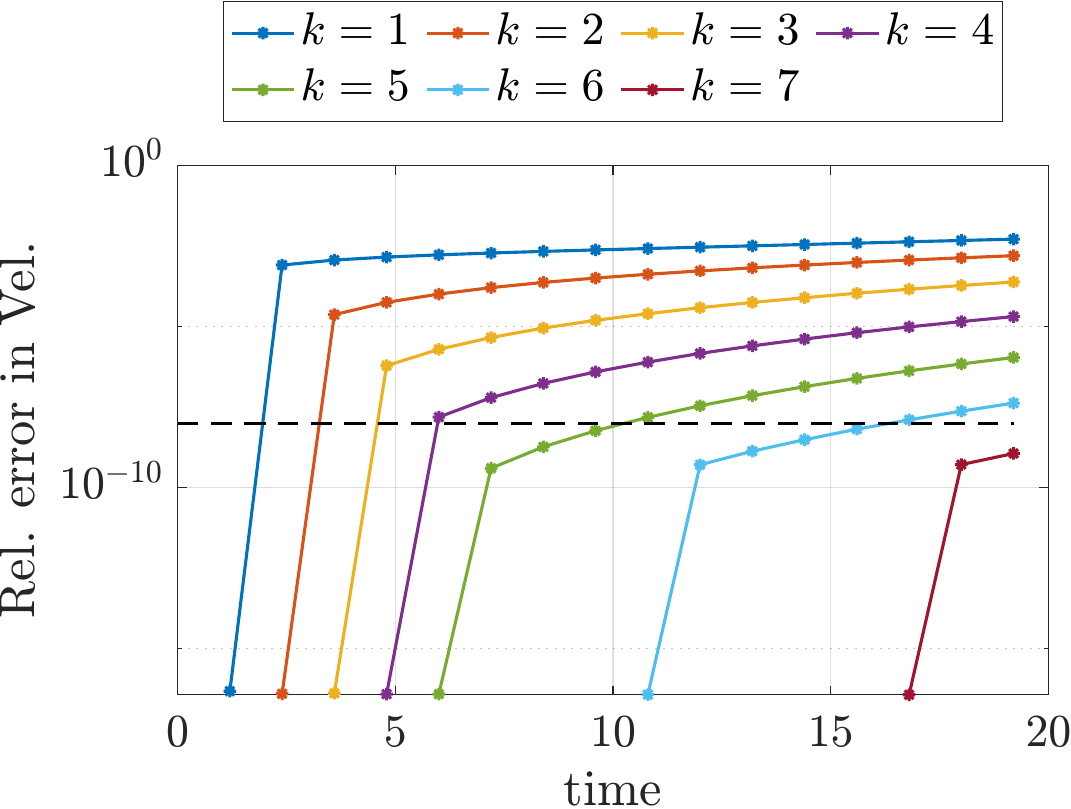}
  }
    \subfigure[Two-stream instability]{
    \includegraphics[width=0.31\columnwidth]{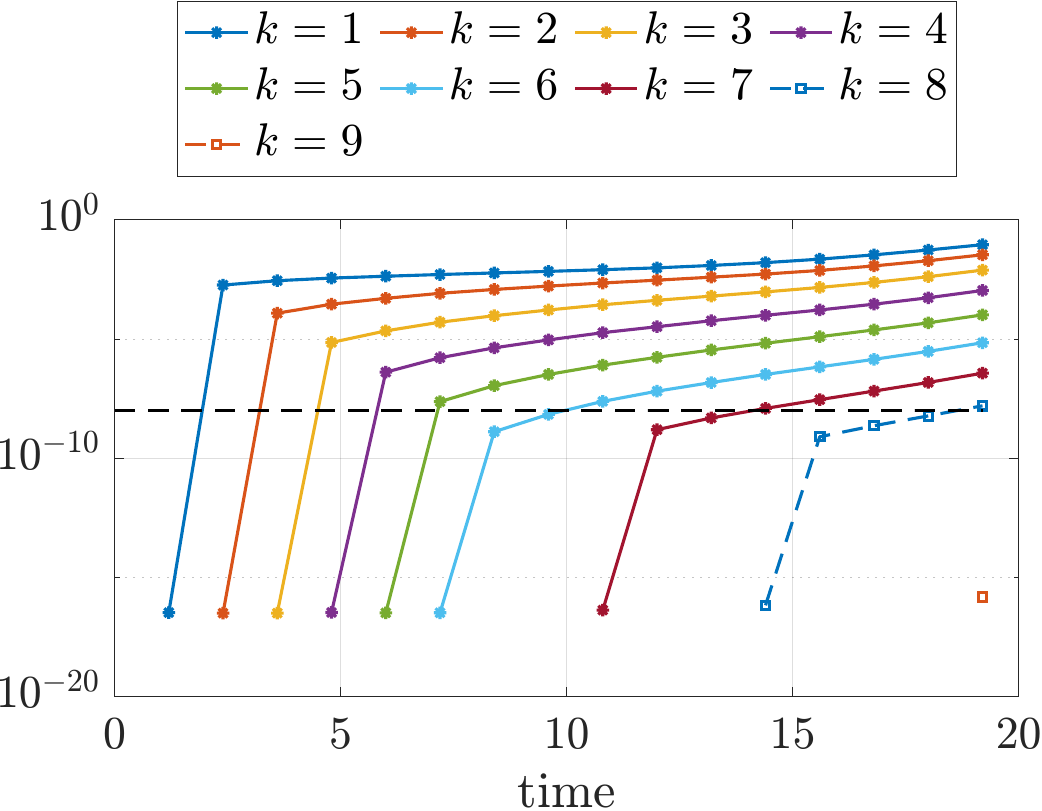}
  }
    \subfigure[Penning trap]{
    \includegraphics[width=0.31\columnwidth]{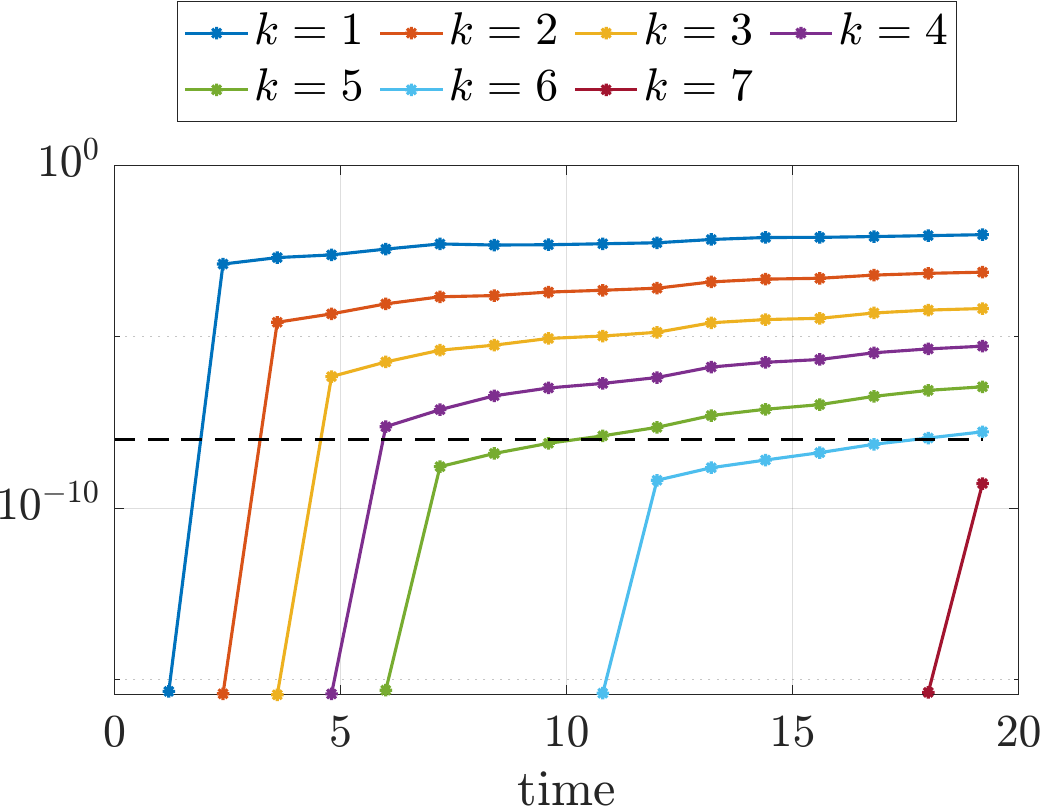}
  }
    \caption{Relative error in the velocity $\vb$ of the particles versus time for the Landau damping (left), the two-stream instability (center) and the Penning trap (right) test cases. The parareal tolerance $10^{-8}$ is marked with the black dashed line.}  
    \figlab{v_error}
\end{figure}

In Figure \figref{v_error} we show the relative error in velocities of the particles based on equation \eqnref{stop_criteria} in each time subdomain for
all the mini-apps. The relative error in positions look similar and hence not shown. We see that the error in the later time subdomains also decreases comparably to the initial subdomains confirming that the
parareal algorithm performs as an approximate Newton's method as explained in \cite{gander2007analysis}. It should be noted that we do not perform a global reduction operation across the time subdomains to check the convergence of the parareal algorithm as this affects the performance. Instead, each time 
subdomain or MPI rank exits the parareal iteration loop as soon as the stopping criterion in equation \eqnref{stop_criteria} is satisfied locally and the previous time subdomain is converged\footnote{This condition is very natural and also needed so that there is a receiving rank for the message sent by the current rank.}. This is the reason for missing data points for some of the time subdomains in later iteration curves in Figure \figref{v_error}.     

We performed a similar study on conservation  as in Figures \figref{landau_twostream_energy}, \figref{charge_momentum_errors} for the coarser time step size of $\Delta t_f=0.05$ and selected a tolerance of 
$10^{-4}$ for the fine propagator NUFFT and $10^{-5}$ for the parareal. We do not repeat those figures for the sake of brevity and to avoid repetition. 

\subsection{Coarse propagator parameter search}
\seclab{heatmaps}
\begin{figure}[h!b!t!]
    \subfigure[$\Delta t_f = 0.05$]{
    \includegraphics[width=0.46\columnwidth]{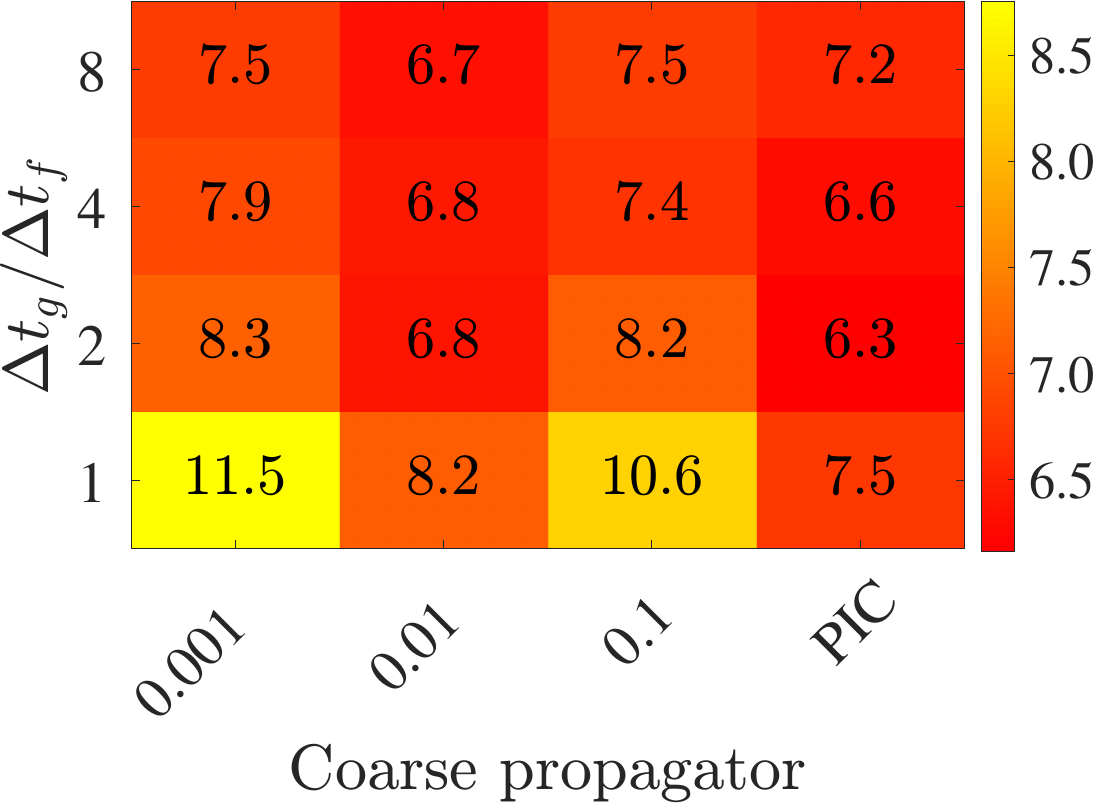}
    \figlab{landau_heatmap_left}
  }
    \subfigure[$\Delta t_f = 0.003125$]{
    \includegraphics[width=0.48\columnwidth]{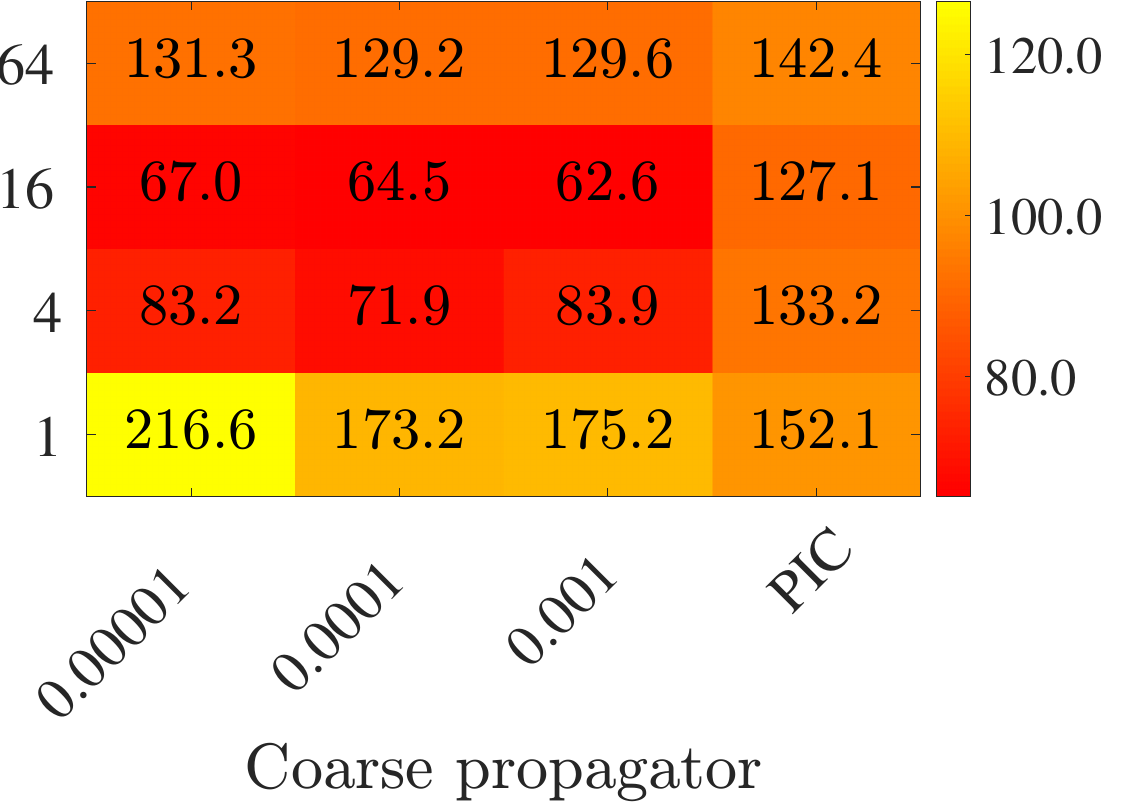}
    \figlab{landau_heatmap_right}
  }
    \caption{Landau damping: Timings of different coarse propagators with different coarse time step sizes in the parareal algorithm for the PIF scheme with $64^3$ modes, $10$ particles per mode. $16$ GPUs are used for the time parallelization, and $\LRc{2,4}$ GPUs are used for the spatial parallelization of $\Delta t_f=0.05$ and  $\Delta t_f=0.003125$ respectively.}
    \figlab{landau_heatmap}
\end{figure}

In Figures \figref{landau_heatmap} and \figref{penning_heatmap} we vary the coarse time step size and the coarse NUFFT tolerance (when PIF is used as the coarse propagator) or use PIC as the coarse propagator and find out the combination which gives the least time to solution for each of the three mini-apps. We select the number of 
GPUs for the spatial parallelization such that the parallel efficiency is greater than fifty percent, i.e., we switch to space-time parallelization and start allocating resources in the time parallelization after the efficiency drops below fifty percent for the spatial parallelization. This is because of the low parallel efficiency of the parareal algorithm, and 
to have an efficiency of $50\%$, parareal needs to converge in two iterations, which is typically not possible in most scenarios. Since the complexity 
of the NUFFT scales as $\mc{O}\LRp{(|log\varepsilon|+1)^dN_p+N_mlogN_m}$ \cite{barnett2019parallel} the smaller the tolerance $\varepsilon$, the higher is the work per MPI rank, and the spatial parallelization scales to more number of MPI ranks or GPUs. For $\Delta t_f=0.05$ and the fine propagator NUFFT tolerance ($\varepsilon$) of $10^{-4}$ we choose 2 GPUs for the spatial parallelization, and for $\Delta t_f=0.003125$, fine propagator $\varepsilon=10^{-7}$, we take 4 GPUs based on the above criterion for 
parallel efficiency. We take the number of GPUs for the time parallelization as $16$. The reference time to solution based on the serial time stepping is around 
$7.8$ and $128$ seconds on $32$ GPUs and $64$ GPUs for $\Delta t_f=0.05$ and $\Delta t_f=0.003125$ respectively for all the mini-apps.      

From Figure \figref{landau_heatmap_left}, for Landau damping, we see that the combination which leads to the least time to solution is PIC as coarse propagator with the ratio of coarse to fine time step sizes as $2$, i.e., $\Delta t_g=0.1$. PIF with $\varepsilon=0.01$ also yields a comparable time to solution for $\Delta t_g/\Delta t_f=2,4$ and $8$. In general, since PIF is a more accurate and costlier coarse propagator than PIC it allows for more coarsening in the time step sizes than PIC for this test case.
For $\Delta t_f=0.003125$, the scenario changes and now PIF with $\varepsilon=0.001$ and coarsening ratio $16$ gives the least time. The other NUFFT tolerances also yield a comparable time for the same coarsening ratio. Since the fine NUFFT tolerance for $\Delta t_f=0.003125$ is $10^{-7}$ we only tested for coarse NUFFT tolerances $10^{-3},10^{-4}$ and $10^{-5}$, as usually values too close to the fine tolerance or too coarse did not result in optimal combinations. We also observe from Figure \figref{landau_heatmap_right} that PIC as coarse propagator takes more than
twice the time as the best combination. This is due to the tighter tolerance of parareal ($10^{-8}$) for this case compared to 
$\Delta t_f=0.05$ ($10^{-5}$) which requires a more accurate coarse propagator. For the
two-stream instability mini-app we performed a similar study and the inferences are mostly similar to that of the Landau damping except that for $\Delta t_f=0.003125$, PIF with $\varepsilon=10^{-4}$ and coarsening ratio $16$ is the best combination. 


\begin{figure}[h!b!t!]
    \subfigure[$\Delta t_f = 0.05$]{
    \includegraphics[width=0.46\columnwidth]{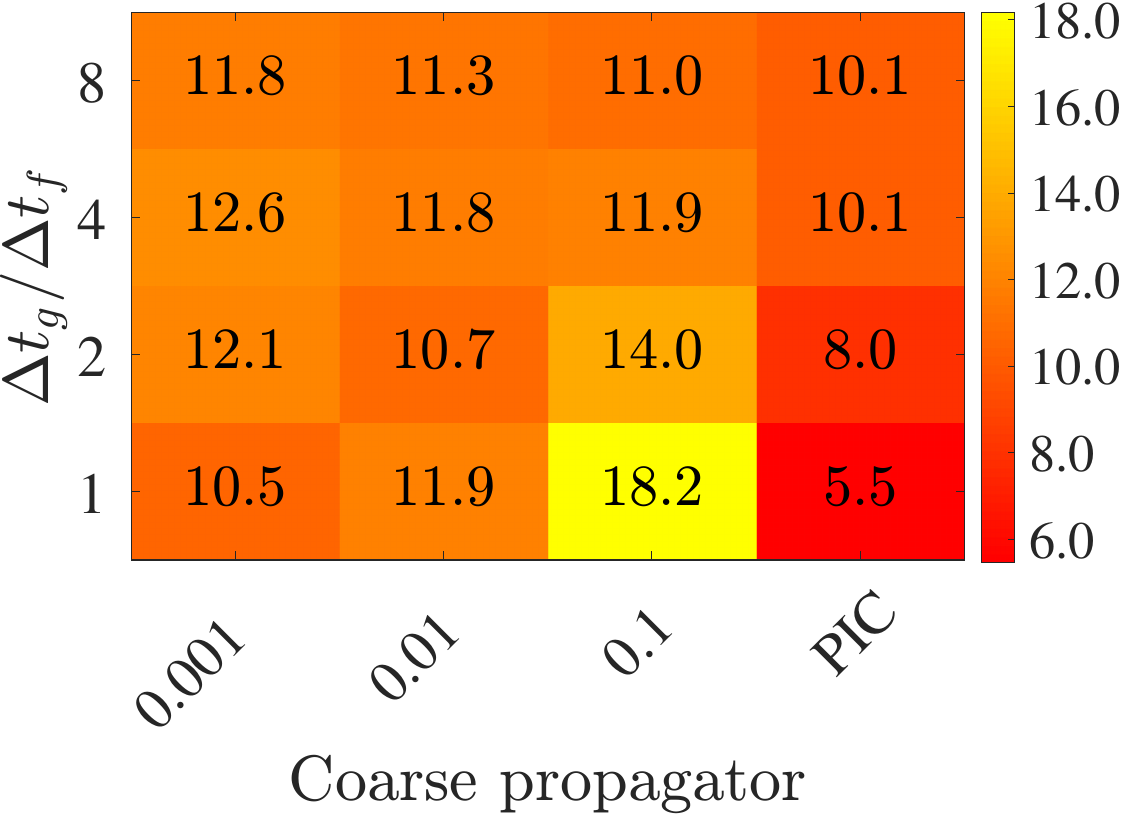}
    \figlab{penning_heatmap_left}
  }
    \subfigure[$\Delta t_f = 0.003125$]{
    \includegraphics[width=0.48\columnwidth]{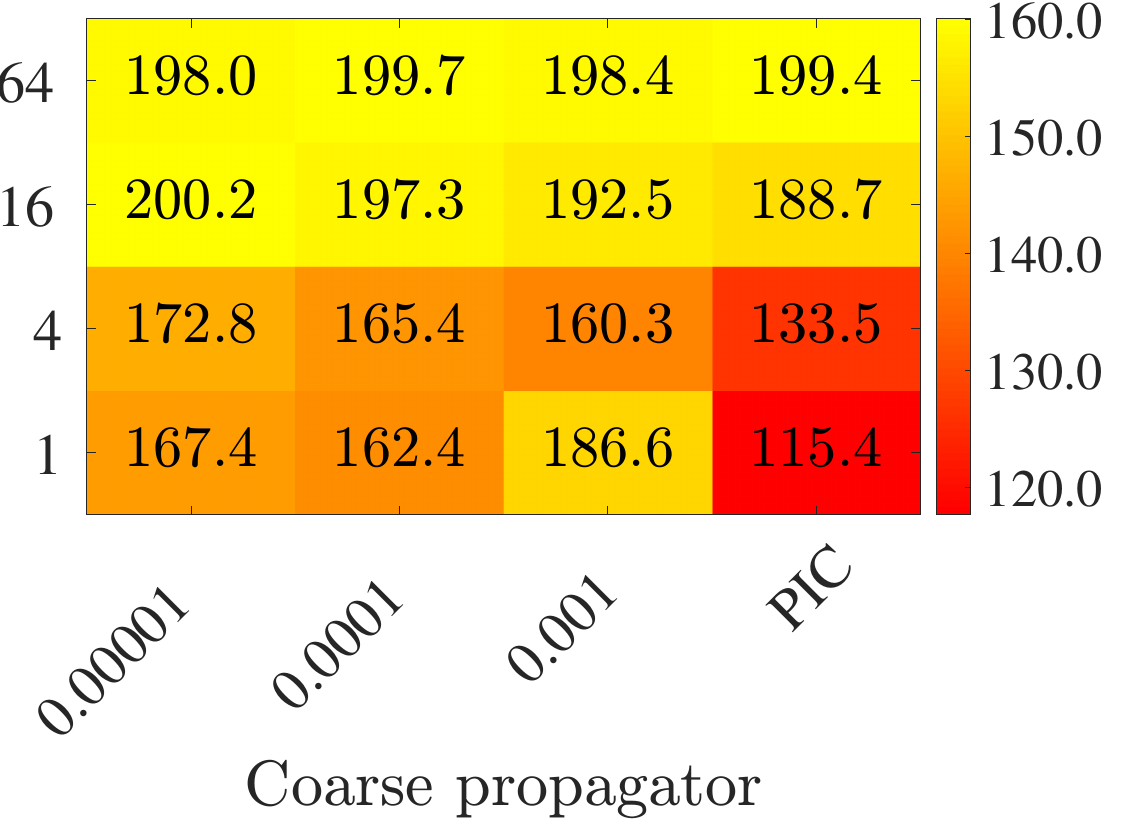}
    \figlab{penning_heatmap_right}
  }
    \caption{Penning trap: Timings of different coarse propagators with different coarse time step sizes in the parareal algorithm for the PIF scheme with $64^3$ modes, $10$ particles per mode. $16$ GPUs are used for the time parallelization and $\LRc{2,4}$ GPUs are used for the spatial parallelization of $\Delta t_f=0.05$ and  $\Delta t_f=0.003125$ respectively.}   
    \figlab{penning_heatmap}
\end{figure}

In the case of the Penning trap, in Figure \figref{penning_heatmap}, we notice that coarsening in time leads to large increase in time to solution which is already 
evident from Figure \figref{conv_penning_long_right} due to the oscillatory nature of the solution. As mentioned in Section \secref{conv_time_step_size}, it is relatively independent of the 
nature of the coarse propagator. In the absence of time coarsening, we would expect the cheapest coarse propagator, i.e., PIC, to give the least time to solution and this is what we see in Figure \figref{penning_heatmap}. We also notice from Figure \figref{penning_heatmap_right} that the least time to solution is still much higher than those for the Landau damping and the two-stream instability test cases. Compared to the reference serial time stepping, we get speedups of $1.2, 1.2$ and $1.4$ for
$\Delta t_f=0.05$, and $2.0,2.0$ and $1.1$ for $\Delta t_f=0.003125$, with the best combinations for the Landau damping, the two-stream instability and the Penning 
trap mini-apps respectively.

\subsection{Multi-block parareal}
\seclab{multiblock}
The convergence and the time to solution of the parareal algorithm can be improved by applying it in multiple blocks or windows instead of in the entire time domain as in \cite{aubanel2011scheduling,nielsen2018communication}. However, this strategy increases the communication time and hence there is a trade-off. Here we adopt that and with $16$ equi-sized blocks for the Penning trap mini-app, using PIC as a coarse propagator for $\Delta t_f=0.003125$,
obtained the time to solution as $82.6$. This
increases the speedup of parareal to $1.5$ compared to the previous 
value of $1.1$ with one block. We also observed that even for the 
Landau damping and the two-stream instability mini-apps, with $\Delta t_f=0.003125$, using PIC as a 
coarse propagator in multiple blocks instead of one block help to 
reduce the time to solution. However, for those cases, it is still not 
able to beat the time to solution obtained with the best combination PIF schemes in the coarse propagators. We also noticed from our experiments, that applying PIF as a coarse propagator in multiple blocks is generally not much beneficial, as the increased initialization and communication costs outweigh the marginal improvement obtained in the iteration counts.

\section*{Reproducibility of computational results}
The source code and data for the simulation is available at \\ https://github.com/srikrrish/ippl/tree/parapif-paper-v1.0.0.

\section*{Acknowledgments}
We would like to acknowledge the computing time in the JUWELS Booster supercomputer provided under the project CSTMA. We would like to thank
Felix Schurk for working on some of the post-processing scripts for the multi-block version of the algorithm for his seminar project. The first author
also would like to thank Antoine Cerfon, Martin Gander and Changxiao Shen for fruitful discussions on this topic. The authors would like to thank the anonymous referees for their critical and useful comments that significantly improved the paper.

\bibliographystyle{siamplain}
\bibliography{references,self}

\begin{thebibliography}{10}

\bibitem{ameres2018stochastic}
{\sc J.~Ameres}, {\em Stochastic and spectral particle methods for plasma
  physics}, PhD thesis, Technische Univerisität München, 2018.

\bibitem{aubanel2011scheduling}
{\sc E.~Aubanel}, {\em Scheduling of tasks in the parareal algorithm}, Parallel
  Computing, 37 (2011), pp.~172--182.

\bibitem{barnes2021finite}
{\sc D.~C. Barnes and L.~Chac{\'o}n}, {\em Finite spatial-grid effects in
  energy-conserving particle-in-cell algorithms}, Computer Physics
  Communications, 258 (2021), p.~107560.

\bibitem{barnett2021aliasing}
{\sc A.~H. Barnett}, {\em Aliasing error of the $\exp(\beta\sqrt{1- z^2})$
  kernel in the nonuniform fast fourier transform}, Applied and Computational
  Harmonic Analysis, 51 (2021), pp.~1--16.

\bibitem{barnett2019parallel}
{\sc A.~H. Barnett, J.~Magland, and L.~af~Klinteberg}, {\em A parallel
  nonuniform fast fourier transform library based on an “exponential of
  semicircle" kernel}, SIAM Journal on Scientific Computing, 41 (2019),
  pp.~C479--C504.

\bibitem{birdsall2004plasma}
{\sc C.~K. Birdsall and A.~B. Langdon}, {\em Plasma physics via computer
  simulation}, CRC press, 2004.

\bibitem{blaum2022}
{\sc K.~Blaum}, {\em Cyclotron frequency in a {P}enning trap}, tech. report,
  Max-Planck Institute for Nuclear Physics, 2022.

\bibitem{briguglio2000parallelization}
{\sc S.~Briguglio, G.~Vlad, B.~Di~Martino, and G.~Fogaccia}, {\em
  Parallelization of plasma simulation codes: gridless finite size particle
  versus particle in cell approach}, Future Generation Computer Systems, 16
  (2000), pp.~541--552.

\bibitem{campos2024variational}
{\sc M.~Campos~Pinto, J.~Ameres, K.~Kormann, and E.~Sonnendr{\"u}cker}, {\em On
  variational fourier particle methods}, Journal of Scientific Computing, 101
  (2024), p.~68.

\bibitem{campos2022variational}
{\sc M.~Campos~Pinto, K.~Kormann, and E.~Sonnendr{\"u}cker}, {\em Variational
  framework for structure-preserving electromagnetic particle-in-cell methods},
  Journal of Scientific Computing, 91 (2022), p.~46.

\bibitem{chen2011energy}
{\sc G.~Chen, L.~Chac{\'o}n, and D.~Barnes}, {\em An energy-and
  charge-conserving, implicit, electrostatic particle-in-cell algorithm},
  Journal of Computational Physics, 230 (2011), pp.~7018--7036.

\bibitem{chen2020semi}
{\sc G.~Chen, L.~Chacon, L.~Yin, B.~J. Albright, D.~J. Stark, and R.~F. Bird},
  {\em A semi-implicit, energy-and charge-conserving particle-in-cell algorithm
  for the relativistic vlasov-maxwell equations}, Journal of Computational
  Physics, 407 (2020), p.~109228.

\bibitem{chen2024fourier}
{\sc Z.~Chen and C.~S. Peskin}, {\em A fourier spectral immersed boundary
  method with exact translation invariance, improved boundary resolution, and a
  divergence-free velocity field}, Journal of Computational Physics, 509
  (2024), p.~113048.

\bibitem{pintwebsite}
{\sc P.~Community}, {\em Parallel-in-time webpage}.
\newblock https://parallel-in-time.org/.
\newblock Accessed: 16-06-2024.

\bibitem{dai2013symmetric}
{\sc X.~Dai, C.~Le~Bris, F.~Legoll, and Y.~Maday}, {\em Symmetric parareal
  algorithms for hamiltonian systems}, ESAIM: Mathematical Modelling and
  Numerical Analysis, 47 (2013), pp.~717--742.

\bibitem{dawson1983particle}
{\sc J.~M. Dawson}, {\em Particle simulation of plasmas}, Reviews of modern
  physics, 55 (1983), p.~403.

\bibitem{dutt1993fast}
{\sc A.~Dutt and V.~Rokhlin}, {\em Fast fourier transforms for nonequispaced
  data}, SIAM Journal on Scientific computing, 14 (1993), pp.~1368--1393.

\bibitem{dutt1995fast}
{\sc A.~Dutt and V.~Rokhlin}, {\em Fast fourier transforms for nonequispaced
  data, ii}, Applied and Computational Harmonic Analysis, 2 (1995),
  pp.~85--100.

\bibitem{eastwood1991virtual}
{\sc J.~W. Eastwood}, {\em The virtual particle electromagnetic particle-mesh
  method}, Computer Physics Communications, 64 (1991), pp.~252--266.

\bibitem{evstatiev2013variational}
{\sc E.~G. Evstatiev and B.~A. Shadwick}, {\em Variational formulation of
  particle algorithms for kinetic plasma simulations}, Journal of Computational
  Physics, 245 (2013), pp.~376--398.

\bibitem{matthias_frey_2024_10878166}
{\sc M.~Frey, A.~Vinciguerra, S.~Muralikrishnan, S.~Mayani, V.~Montanaro,
  M.~Sadr, A.~Adelmann, M.~Winkler, and F.~Schurk}, {\em Ippl-framework/ippl:
  Ippl-3.2.0}, Mar. 2024.

\bibitem{gander201550}
{\sc M.~J. Gander}, {\em 50 years of time parallel time integration}, in
  Multiple Shooting and Time Domain Decomposition Methods: MuS-TDD, Heidelberg,
  May 6-8, 2013, Springer, 2015, pp.~69--113.

\bibitem{gander2013paraexp}
{\sc M.~J. Gander and S.~G{\"u}ttel}, {\em Paraexp: A parallel integrator for
  linear initial-value problems}, SIAM Journal on Scientific Computing, 35
  (2013), pp.~C123--C142.

\bibitem{gander2008nonlinear}
{\sc M.~J. Gander and E.~Hairer}, {\em Nonlinear convergence analysis for the
  parareal algorithm}, in Domain decomposition methods in science and
  engineering XVII, Springer, 2008, pp.~45--56.

\bibitem{gander2020paradiag}
{\sc M.~J. Gander, J.~Liu, S.-L. Wu, X.~Yue, and T.~Zhou}, {\em Paradiag:
  Parallel-in-time algorithms based on the diagonalization technique}, arXiv
  preprint arXiv:2005.09158,  (2020).

\bibitem{gander2023unified}
{\sc M.~J. Gander, T.~Lunet, D.~Ruprecht, and R.~Speck}, {\em A unified
  analysis framework for iterative parallel-in-time algorithms}, SIAM Journal
  on Scientific Computing, 45 (2023), pp.~A2275--A2303.

\bibitem{gander2007analysis}
{\sc M.~J. Gander and S.~Vandewalle}, {\em Analysis of the parareal
  time-parallel time-integration method}, SIAM Journal on Scientific Computing,
  29 (2007), pp.~556--578.

\bibitem{gotschel2020twelve}
{\sc S.~G{\"o}tschel, M.~Minion, D.~Ruprecht, and R.~Speck}, {\em Twelve ways
  to fool the masses when giving parallel-in-time results}, in Workshops on
  Parallel-in-Time Integration, Springer, 2020, pp.~81--94.

\bibitem{he2016hamiltonian}
{\sc Y.~He, Y.~Sun, H.~Qin, and J.~Liu}, {\em Hamiltonian particle-in-cell
  methods for {V}lasov-{M}axwell equations}, Physics of Plasmas, 23 (2016).

\bibitem{hockney2021computer}
{\sc R.~W. Hockney and J.~W. Eastwood}, {\em Computer simulation using
  particles}, CRC Press, 2021.

\bibitem{huang2016finite}
{\sc C.-K. Huang, Y.~Zeng, Y.~Wang, M.~D. Meyers, S.~Yi, and B.~J. Albright},
  {\em Finite grid instability and spectral fidelity of the electrostatic
  particle-in-cell algorithm}, Computer Physics Communications, 207 (2016),
  pp.~123--135.

\bibitem{jianyuan2018structure}
{\sc X.~Jianyuan, Q.~Hong, and L.~Jian}, {\em Structure-preserving geometric
  particle-in-cell methods for {V}lasov-{M}axwell systems}, Plasma Science and
  Technology, 20 (2018), p.~110501.

\bibitem{kraus2017metriplectic}
{\sc M.~Kraus and E.~Hirvijoki}, {\em Metriplectic integrators for the landau
  collision operator}, Physics of Plasmas, 24 (2017).

\bibitem{kraus2017gempic}
{\sc M.~Kraus, K.~Kormann, P.~J. Morrison, and E.~Sonnendr{\"u}cker}, {\em
  {GEMPIC:} geometric electromagnetic particle-in-cell methods}, Journal of
  Plasma Physics, 83 (2017), p.~905830401.

\bibitem{langdon1970theory}
{\sc A.~B. Langdon and C.~K. Birdsall}, {\em Theory of plasma simulation using
  finite-size particles}, The Physics of Fluids, 13 (1970), pp.~2115--2122.

\bibitem{lapenta2017exactly}
{\sc G.~Lapenta}, {\em Exactly energy conserving semi-implicit particle in cell
  formulation}, Journal of Computational Physics, 334 (2017), pp.~349--366.

\bibitem{lewis1970energy}
{\sc H.~R. Lewis}, {\em Energy-conserving numerical approximations for {V}lasov
  plasmas}, Journal of Computational Physics, 6 (1970), pp.~136--141.

\bibitem{lions2001resolution}
{\sc J.-L. Lions, Y.~Maday, and G.~Turinici}, {\em R{\'e}solution d'edp par un
  sch{\'e}ma en temps {\guillemotleft}parar{\'e}el{\guillemotright}}, Comptes
  Rendus de l'Acad{\'e}mie des Sciences-Series I-Mathematics, 332 (2001),
  pp.~661--668.

\bibitem{low1958lagrangian}
{\sc F.~Low}, {\em A {L}agrangian formulation of the {Boltzmann-Vlasov}
  equation for plasmas}, Proceedings of the Royal Society of London. Series A.
  Mathematical and Physical Sciences, 248 (1958), pp.~282--287.

\bibitem{markidis2011energy}
{\sc S.~Markidis and G.~Lapenta}, {\em The energy conserving particle-in-cell
  method}, Journal of Computational Physics, 230 (2011), pp.~7037--7052.

\bibitem{mitchell2019efficient}
{\sc M.~S. Mitchell, M.~T. Miecnikowski, G.~Beylkin, and S.~E. Parker}, {\em
  Efficient fourier basis particle simulation}, Journal of Computational
  Physics, 396 (2019), pp.~837--847.

\bibitem{muralikrishnan2021sparse}
{\sc S.~Muralikrishnan, A.~J. Cerfon, M.~Frey, L.~F. Ricketson, and
  A.~Adelmann}, {\em Sparse grid-based adaptive noise reduction strategy for
  particle-in-cell schemes}, Journal of Computational Physics: X, 11 (2021),
  p.~100094.

\bibitem{muralikrishnan2024scaling}
{\sc S.~Muralikrishnan, M.~Frey, A.~Vinciguerra, M.~Ligotino, A.~J. Cerfon,
  M.~Stoyanov, R.~Gayatri, and A.~Adelmann}, {\em Scaling and performance
  portability of the particle-in-cell scheme for plasma physics applications
  through mini-apps targeting exascale architectures}, in Proceedings of the
  2024 SIAM Conference on Parallel Processing for Scientific Computing (PP),
  SIAM, 2024, pp.~26--38.

\bibitem{nielsen2018communication}
{\sc A.~S. Nielsen, G.~Brunner, and J.~S. Hesthaven}, {\em Communication-aware
  adaptive parareal with application to a nonlinear hyperbolic system of
  partial differential equations}, Journal of Computational Physics, 371
  (2018), pp.~483--505.

\bibitem{nievergelt1964parallel}
{\sc J.~Nievergelt}, {\em Parallel methods for integrating ordinary
  differential equations}, Communications of the ACM, 7 (1964), pp.~731--733.

\bibitem{ohana2016towards}
{\sc N.~Ohana, A.~Jocksch, E.~Lanti, T.~Tran, S.~Brunner, C.~Gheller,
  F.~Hariri, and L.~Villard}, {\em Towards the optimization of a gyrokinetic
  particle-in-cell (pic) code on large-scale hybrid architectures}, in Journal
  of Physics: Conference Series, vol.~775(1), IOP Publishing, 2016, p.~012010.

\bibitem{ong2020applications}
{\sc B.~W. Ong and J.~B. Schroder}, {\em Applications of time parallelization},
  Computing and Visualization in Science, 23 (2020), pp.~1--15.

\bibitem{pinto2022semi}
{\sc M.~C. Pinto and V.~Pag{\`e}s}, {\em A semi-implicit electromagnetic
  fem-pic scheme with exact energy and charge conservation}, Journal of
  Computational Physics, 453 (2022), p.~110912.

\bibitem{potts2001fast}
{\sc D.~Potts, G.~Steidl, and M.~Tasche}, {\em Fast fourier transforms for
  nonequispaced data: A tutorial}, Modern Sampling Theory: Mathematics and
  Applications,  (2001), pp.~247--270.

\bibitem{ricketson2016sparse}
{\sc L.~F. Ricketson and A.~J. Cerfon}, {\em Sparse grid techniques for
  particle-in-cell schemes}, Plasma Physics and Controlled Fusion, 59 (2016),
  p.~024002.

\bibitem{shadwick2014variational}
{\sc B.~A. Shadwick, A.~B. Stamm, and E.~G. Evstatiev}, {\em Variational
  formulation of macro-particle plasma simulation algorithms}, Physics of
  Plasmas, 21 (2014).

\bibitem{shen2024particle}
{\sc C.~N. Shen, A.~Cerfon, and S.~Muralikrishnan}, {\em A particle-in-fourier
  method with semi-discrete energy conservation for non-periodic boundary
  conditions}, Journal of Computational Physics, 519 (2024), p.~113390.

\bibitem{shih2021cufinufft}
{\sc Y.-h. Shih, G.~Wright, J.~And{\'e}n, J.~Blaschke, and A.~H. Barnett}, {\em
  cufinufft: a load-balanced gpu library for general-purpose nonuniform ffts},
  in 2021 IEEE International Parallel and Distributed Processing Symposium
  Workshops (IPDPSW), IEEE, 2021, pp.~688--697.

\bibitem{squire2012geometric}
{\sc J.~Squire, H.~Qin, and W.~M. Tang}, {\em Geometric integration of the
  {V}lasov-{M}axwell system with a variational particle-in-cell scheme},
  Physics of Plasmas, 19 (2012).

\bibitem{tretiak2019arbitrary}
{\sc K.~Tretiak and D.~Ruprecht}, {\em An arbitrary order time-stepping
  algorithm for tracking particles in inhomogeneous magnetic fields}, Journal
  of computational physics: X, 4 (2019), p.~100036.

\bibitem{webb2016spectral}
{\sc S.~D. Webb}, {\em A spectral canonical electrostatic algorithm}, Plasma
  Physics and Controlled Fusion, 58 (2016), p.~034007.

\bibitem{xiao2019structure}
{\sc J.~Xiao and H.~Qin}, {\em Structure-preserving geometric particle-in-cell
  algorithm suppresses finite-grid instability--comment on" finite grid
  instability and spectral fidelity of the electrostatic particle-in-cell
  algorithm''by {H}uang et al}, arXiv preprint arXiv:1904.00535,  (2019).

\end{thebibliography}
\end{document}